\documentclass[10pt]{amsart}
\usepackage{amsmath,amssymb,amsthm,mathtools}
\usepackage{newtxtext,newtxmath}
\usepackage{geometry}
\usepackage{microtype}
\usepackage{xcolor}
\usepackage{bm}
\usepackage{tikz}
\usepackage{soul}

\usepackage[colorlinks=true,linkcolor=blue,citecolor=blue,urlcolor=blue]{hyperref}

\allowdisplaybreaks
\theoremstyle{plain}
\newtheorem{theorem}{Theorem}[section]

\newtheorem{lemma}[theorem]{Lemma}
\newtheorem{corollary}[theorem]{Corollary}
\theoremstyle{remark}

\numberwithin{equation}{section}
\newcommand{\sinc}{\operatorname{sinc}}
\newcommand{\R}{\mathbb R}

\newcommand{\PW}{\mathrm{PW}}
\newcommand{\dd}{\,\mathrm d}
\newcommand{\ee}{\,\mathrm e}
\newcommand{\ii}{\,\mathrm i}

\newcommand{\norm}[1]{\left\lVert#1\right\rVert}

\newcommand{\supp}{\operatorname{supp}}
\newcommand{\dist}{\operatorname{dist}}

\author[A. Bondarenko]{Andriy Bondarenko}
\address{Department of Mathematical Sciences \\ Norwegian University of Science and Technology \\ NO-7491 Trondheim \\ Norway}
\email{andriybond@gmail.com}

\author[J. Ortega-Cerd\`{a}]{Joaquim Ortega-Cerd\`{a}}
\address{Department de Matem\`{a}tiques i Inform\`{a}tica, Universitat de
Barcelona, Gran Via 585, 08007 Barcelona, Spain and
CRM, Centre de Recerca Matemàtica, Campus de Bellaterra Edifici C, 08193
Bellaterra, Barcelona, Spain}
\email{jortega@ub.edu}

\author[K. Seip]{Kristian Seip}
\address{Department of Mathematical Sciences, Norwegian University of Science and Technology (NTNU), 7491 Trondheim, Norway}
\email{kristian.seip@ntnu.no}

\thanks{Bondarenko and Seip were supported in part by Grant 334466 of the
Research Council of Norway. Ortega-Cerd\`a has been supported by grant PID2024-160033NB-I00 of the Agencia Estatal de Investigación and by the 2024 ICREA 00142 grant of the Generalitat de Catalunya.}
\title[The zeros of point evaluation extremals in Paley--Wiener spaces]{The zeros of point evaluation extremals in Paley--Wiener spaces}
\date{\today}
\hypersetup{pdftitle={Some Analytic Properties of Extremal Functions in PWp Spaces},pdfsubject={Shifted-lattice zero asymptotics and Fourier regularity for Paley--Wiener extremals}}
\begin{document}
\begin{abstract}
For $1\le p < \infty$, $\varphi_p$ is the function of minimal $L^p$ norm among all functions $f$ of exponential type at most $\pi$ for which $f(0)=1$. We start from the ansatz that the positive zeros $\tau_n$ of $\varphi_p$  are close to \(n+1/p-1/2\) and use a weighted convex correction to first place them in an \(\ell^2\)-neighborhood of this sequence. A bootstrap argument then improves this to a {Poincar\'e asymptotic} expansion of $\tau_n-(n+\frac1p-\frac12)$ in terms of odd powers of $(n+\frac1p-\frac12)^{-1}$.  We also get precise information about the Fourier transform $\widehat{\varphi}_p$, a result that in particular implies that for all integers $n\ge 2$, $\widehat{\varphi}_{2n}$ is an integrable function which, when convolved with itself $2n-2$ times, is constant on the interval $[-\pi,\pi]$. 
In particular, we identify the extremizer in an
 autoconvolution problem posed by B. Green in 2001.

\end{abstract}
\maketitle

\section{Introduction and main results}
Following \cite{BCOS}, we let $\varphi_p$ for $1\le p <\infty$ denote  
the unique solution to the extremal problem
\begin{equation}\label{eq:extremalproblem}
	\frac{1}{\mathscr{C}_p} = \inf_{f \in PW^p} \left\{\|f\|_{L^p(\mathbb R)}^p \,: \, f(0) = 1\right\},
\end{equation}
where
$\mathscr{C}_p$  is the smallest positive constant $C$ such that the inequality 
\begin{equation}\label{eq:pointeval}
	|f(0)|^p \leq C \|f\|_{L^p(\mathbb R)}^p
\end{equation}
holds for every $f$ in the Paley--Wiener space $PW^p$, i.e., the subspace of $L^p(\mathbb{R})$ consisting of entire functions of exponential type at most $\pi$. We have trivially $\varphi_2(x)=\sinc(\pi x)$, where as usual $\sinc x := \frac{\sin{x}}{x}$, while $\varphi_1$ is the H\"{o}rmander--Bernhardsson function that was recently characterized in \cite{BORS}. It is known from \cite{BCOS} that for general $p$,
\begin{equation} \label{eq:phip} \varphi_p(x)=\prod_{n=1}^\infty \left(1-\frac{x^2}{\tau_n(p)^2}\right), \end{equation}
with $\tau_n(p)$ an increasing sequence of positive numbers. 

The remarkable analytic structure underpinning the H\"{o}rmander--Bernhardsson function $\varphi_1$ led in \cite{BORS} to the expansion $\tau_n(1)=n+\tfrac12-\sigma((n+\frac12)^{-1})$, where $\sigma(z)$ is an odd power series with nonnegative coefficients and domain of convergence $|z|\le 2$. The main result of this paper verifies that the zeros $\tau_n(p)$ enjoy a similar behavior for $1<p<\infty$.
\begin{theorem}\label{thm:main}
Let $1<p<\infty$.  There exist real numbers
$a_{p,0},a_{p,1},\ldots$ such that, for every integer $M\geq0$,
\begin{equation}\label{eq:complete-zero-Poincare}
 \tau_n=n+\tfrac1p-\tfrac12+\sum_{k=0}^{M}\frac{a_{p,k}}{(n+\frac1p-\frac12)^{2k+1}}
 + O_M\left( \frac{1}{(n+\frac1p-\frac12)^{2M+3}}\right).
\end{equation}
\end{theorem}

In the general case $1<p<\infty$, $p\neq 2$, only rather rudimentary results about $\tau_n(p)$ are known from earlier work. See for example Theorem 1.4, Theorem 1.5, and Theorem 5.1 of \cite{BCOS}. Previous studies have mainly focused on estimating $\mathscr{C}_p$, which, as a curious by-product of our proof of Theorem~\ref{thm:main}, can be expressed as 
\begin{equation}
    \label{eq:Cpformula}
 \frac1{\mathscr C_p}=
 \frac{2\pi^{3/2}\Gamma\!\left(\frac{p+1}{2}\right)}
      {\Gamma\!\left(\frac p2\right)}
 \left[
 \frac{\Gamma\!\left(\frac12+\frac1p\right)^2}{\pi}
 \right]^p
 \prod_{n=1}^{\infty}
 \left(
 \frac{n+\frac1p-\frac12}{\tau_n(p)}
 \right)^{2p}, \qquad 1\le p < \infty,
\end{equation}
in agreement with what was found in \cite{BORS} in the case $p=1$.
Levin and Lubinsky \cite{LevinLubinsky} found that $\mathscr{C}_p^{-1}$ appears as a scaling limit for certain Christoffel functions. See also Lubinsky's survey \cite{LubinskyScaling}, where Problems 4 and 5 specifically ask for estimates of $\mathscr{C}_p$ and for information about the extremizers $\varphi_p$. Various bounds for $\mathscr{C}_p$  were found in \cite{BCOS} and more recently by Instanes in \cite{Instanes}. Such results in turn imply estimates for the Nikolskii constants for trigonometric polynomials by \cite[Theorem~1]{GorbachevMartyanov}. As we will see below, a completely different origin of the problem of computing $\mathscr{C}_4$ can be found in work of Green in combinatorial number theory \cite{Green2001, GreenProblems}. 

Theorem~\ref{thm:main} verifies \cite[Conjecture 9.9]{BCOS} in the range $p\ge1$, and also proves \cite[Conjecture 9.12]{BCOS}. To record another resolution of a problem from \cite{BCOS}, we introduce the Fourier normalization
\[
 \widehat f(\xi)=\int_\R f(x)e^{-ix\xi}\dd x,\qquad
 f(x)=\frac1{2\pi}\int_\R\widehat f(\xi)e^{ix\xi}\dd\xi
\]
whenever the integrals are classically defined.  For general functions and all distributions occurring below, the Fourier transform is understood in the tempered-distribution sense with the same normalization. For \(s\) in \(\R\), \(H^s(\R)\) denotes the usual Sobolev space, with
\[
 \norm{q}_{H^s}^2\coloneqq \frac1{2\pi}\int_\R(1+|t|^2)^s|\widehat q(t)|^2\dd t
\]
when the right-hand side is finite. Since functions in $H^s(\R)$ are continuous when $s>\tfrac12$, the following corollary is a stronger assertion than \cite[Conjecture 9.13]{BCOS}.

\begin{corollary}\label{cor:Fourier}
Fix $1\le p < \infty$. There exists a function $h_p$ in $\bigcap_{s<\frac2p+\frac12}H^{s}(\R)$ with $\supp h_p \subset [-\pi,\pi]$ such that
\[ \widehat{\varphi}_p(\xi)=A(p)\big(\cos \frac{\xi}{2} \big)^{\frac{2}{p}-1}+ h_p(\xi), \quad {|\xi|<\pi},\]
where 
\[ A(p) \coloneqq \frac{2^{\frac2p-1}\Gamma^2(\frac1p+\frac12)}{\Gamma(\frac2p)}\prod_{n\ge 1} \frac{(n+\tfrac1p-\tfrac12)^2}{\tau_n(p)^2}.\]
\end{corollary}
An interesting consequence of Corollary~\ref{cor:Fourier} is that it leads to a solution 
of the following problem. We define the convolution of $\widehat{f}$ and $\widehat{g}$ as follows:
\[\widehat{f}\ast\widehat{g}\,(\eta) := \int_{-\infty}^\infty \widehat{f}(\xi)\,\widehat{g}(\eta-\xi)\,\frac{d\xi}{2\pi}.\]
If $\widehat{\psi} \in L^1(\mathbb R)$ is supported on $[-\pi,\pi]$ and continuous on $(-\pi,\pi)$, we declare it to be an admissible function. We let the iterated convolution operator $\operatorname{C}_n$ act on admissible functions $\widehat{\psi}$ so that
\[\operatorname{C}_0\widehat{\psi} := \widehat{\psi}, \qquad \operatorname{C}_1\widehat{\psi} := \widehat{\psi} \ast \widehat{\psi}, \qquad \operatorname{C}_2\widehat{\psi} = \widehat{\psi} \ast \widehat{\psi} \ast \widehat{\psi}, \qquad \operatorname{C}_3\widehat{\psi} := \widehat{\psi} \ast \widehat{\psi} \ast \widehat{\psi} \ast \widehat{\psi}, \qquad\ldots\]
In \cite[Sec. 9]{BCOS}, it was asked whether for every positive integer $k$, there exists an admissible function $\widehat{\psi}$ such that 
\begin{equation}\label{eq:conveq} 
	\operatorname{C}_{2k}{\widehat{\psi}}(\xi) = 1 
\end{equation}
for every $-\pi\leq \xi \leq \pi$. The following corollary is immediate from the $L^p$ orthogonality condition (see \eqref{eq:Lportho} below) and Corollary~\ref{cor:Fourier}.
\begin{corollary}
For every $k\ge 1$, $\widehat{\varphi}_{2k+2}$  is a constant times an admissible solution to \eqref{eq:conveq}.
\end{corollary}
We will in a forthcoming paper (see below) establish positivity of $\widehat{\varphi}_p$ for $1<p$. Therefore, by Plancherel's identity, $\widehat{\varphi}_4$ yields the extremizer of \cite[Problem 35]{GreenProblems}, which, in our normalization, requires us to compute
\[ \inf\Big\{ \norm{C_1 h}_{L^2(\R)}^2: \ \supp h\subset [-\pi,\pi],\ h\in L^1(\R), \ h(\xi)\ge 0, \ \int_{-\pi}^\pi h(\xi) d\xi=2\pi \Big\}. \] In this normalization, the infimum equals $2\pi\mathscr C_4^{-1}$.
This problem, first stated in \cite{Green2001} without the constraint $h\ge 0$, has been considered by many authors, see for instance the latest contribution~\cite{Rechnetzer}. 

To close this introduction, we give an outline of our strategy for proving Theorem~\ref{thm:main}. Our starting point is the $L^p$ orthogonality condition from \cite[Thm. 3.6]{BCOS}: We have $\varphi=\varphi_p$ for real even $\varphi \in PW^p$, with $p\ge 1$, if and only if $\varphi(0)=1$ and 
\begin{equation} \label{eq:Lportho}   \int_{-\infty}^\infty f(x) |\varphi(x)|^{p-2} \varphi(x) dx =0 \end{equation}
holds for every even function $f$  in $PW^p$ satisfying $f(0)=0$. This condition was used successfully in \cite{BORS, BORSprelim}, with arguments relying crucially on the fact that the ``kernel'' $|\varphi|^{p-2}\varphi$ takes values $\pm 1$ when $p=1$. In the present less favorable case, we will take an alternative approach. 

From this point on, when $p$ is fixed, we suppress the dependence on $p$ and write just $\tau_n=\tau_n(p)$. We set $\gamma\coloneqq \frac1p-\frac12$ and start from the ansatz that the gamma quotient
\begin{equation}\label{eq:main-Gamma-product}
 E_p(z)\coloneqq \frac{\Gamma(1+\gamma)^2}{\Gamma(1+\gamma+z)\Gamma(1+\gamma-z)}
 =\prod_{n\ge1}\left(1-\frac{z^2}{c_n^2}\right), \quad c_n\coloneqq n+\gamma,
\end{equation}
is a good initial approximation of $\varphi_p$. By a computation attributed to Ramanujan \cite[Section 7.8]{Titchmarsh},
\[ \widehat{E}_p(\xi)= \frac{2^{\frac2p-1}\Gamma^2(\frac1p+\frac12)}{\Gamma(\frac2p)} \big(\cos \frac{\xi}{2} \big)^{\frac{2}{p}-1},\ |\xi| <\pi, \]
and so Corollary~\ref{cor:Fourier} asserts that the scaled function
\[ K(p)E_p(x), \quad K(p)\coloneqq \prod_{n\ge 1} \frac{c_n^2}{\tau_n^2},\]
differs from $\varphi_p$ by a function with faster decay than $E_p$ (see Lemma~\ref{lem:Gamma-properties} below). 
The first step in our proof of Theorem~\ref{thm:main} verifies the following weaker version of this deviation. 
\begin{lemma}\label{prop:Gamma-correction}
For $p>1$ the function $r_p\coloneqq \varphi_p - K(p) E_p$
satisfies the integrability condition
\begin{equation}\label{eq:hilbert} \int_0^\infty |r_p(x)|^2 (1+x^2)^{\frac2p} \dd x < \infty .\end{equation}
\end{lemma}
It follows from this lemma that the sequence $(\tau_n-c_n)$ belongs to $\ell^2$ (see Lemma~\ref{cor:l2-localization} below). The latter result is parallel to that obtained in \cite{BORSprelim}, on which our subsequent study of the H\"{o}rmander--Bernhardsson function \cite{BORS} was based. However, our technique will be 
different from that of \cite{BORSprelim} which relied on a fixed point argument. In the present case, we begin by establishing the basic properties of $E_p$ and of the Fourier transform of $x^2|E_p(x)|^{p-2}E_p(x)$. This analysis leads, via \eqref{eq:Lportho}, to another extremal problem, formulated in a weighted Hilbert space. This problem is solved by a basic result of convex analysis, and its solution yields $\varphi_p$ as a perturbation of $E_p$. This Hilbert space approach leads to the $L^2$ bound in \eqref{eq:hilbert}.

The fact that $(\tau_n-c_n)$ is in $\ell^2$  will be the starting point for a bootstrap argument that will take us to Theorem~\ref{thm:main}. 
This part of our proof relies on a reformulation of \eqref{eq:Lportho}, namely that an infinite product of the form
\[ \varphi(x)=\prod_{n=1}^\infty \left(1-\frac{x^2}{t_n^2}\right), \]
with $(t_n)$ an increasing sequence of positive numbers such that $(t_n-c_n)$ is in $\ell^2$, will satisfy $\varphi=\varphi_p$ if and only if
\begin{equation} \label{eq:reform} \int_{\R} |\varphi(x)|^p \frac{x^2}{t_k^2-x^2} \dd x =0, \qquad k=1,2,\ldots . \end{equation}
The closeness of $E_p$ to $\varphi_p$ is now embodied in the integral in \eqref{eq:reform} being $O(k^{-2})$ for $\varphi=E_p$. This observation leads to the idea that one may gain more information by studying the mapping that sends $(t_n)$ to the sequence defined by the integral in \eqref{eq:reform}. The bootstrap argument is based on a careful study of the rather remarkable properties of this mapping, or rather a suitably weighted version of it. A key point will be that its Jacobian is a diagonal matrix when  $t_n=\tau_n$ and that, for small deviations from the extremal sequence $(\tau_n)$, its off-diagonal is essentially a Hilbert matrix. The fine properties of this Jacobian matrix, along with the initial information about $(\tau_n-c_n)$ and an expansion of the integral in \eqref{eq:reform} for $t_n=c_n$, will form the basis for our bootstrap argument.

As usual, constants denoted by $C_p,c_p$ may change from line to line, and $A\asymp_p B$ means that both $A\le C_pB$ and $B\le C_pA$ hold. 

\subsection*{Forthcoming work} In the course of establishing the results of this paper, we found, with the help of GPT-6 Astra, an alternative method that yields much more information about the function $\varphi_p$ and its Fourier transform $\widehat{\varphi}_p$. This new approach also allows us to treat the non-convex range $0<p<1$ and, in particular, to establish uniqueness of the extremizer in this case as well. In addition, we have been able to clarify how the H\"{o}rmander--Bernhardsson function, see \cite{HB1993}, fits into the wider context of $PW^p$ for $0<p<\infty$. We are in the process of preparing another paper in which this new approach is presented and explored. Let us point out explicitly that positivity of $\widehat{\varphi}_p$ will be proved when $p>1$, and so we will verify that $\widehat{\varphi}_4$ in particular has the properties required to solve \cite[Problem 35]{GreenProblems}. 

We are of the opinion that, in spite of this new development, the method used in the present paper is of independent interest. First of all, it is close in spirit to the approach of the initial paper \cite{BCOS} that established some of the basic properties of the point evaluation extremizers for $PW^p$, and so reflects how our understanding of the problem has evolved. In addition, we hope that the idea of transforming the problem from $L^p$ to a Hilbert space, followed by a bootstrap argument, could be of interest in a wider context. For these reasons, we have decided not to abandon the present paper. 

\section{From the gamma quotient to \texorpdfstring{$\varphi_p$}{varphi	extunderscore p}: Proof of Lemma~\ref{prop:Gamma-correction}}
\label{sec:Gamma-correction}

We begin by recording some basic facts about the gamma quotient $E_p$. To this end, we will also need the weighted ratio of gamma functions
\begin{equation}\label{eq:A-m}
 m_p(x)\coloneqq x^\alpha
 \frac{\Gamma(x-\gamma)}{\Gamma(x+\gamma+1)}, \quad x>\max\{0,\gamma\}, 
\end{equation}
where we have set $\alpha\coloneqq \frac2p$.
\begin{lemma}
\label{lem:Gamma-properties}
For \({x>\max\{0,\gamma\}}\),
\begin{equation}\label{eq:Gamma-phase}
 x^\alpha E_p(x)
 =\frac{\Gamma(1+\gamma)^2}{\pi} m_p(x)\sin\pi(x-\gamma).
\end{equation}
Moreover, for $x\ge 1$ and for every integer \(j\ge0\),
\begin{equation}\label{eq:m-symbol}
 \left|\frac{\dd^j}{\dd x^j}\{m_p(x)-1\}\right|
 \le C_{j,p}x^{-2-j}.
\end{equation}
\end{lemma}

\begin{proof}

The first assertion \eqref{eq:Gamma-phase} follows from the reflection formula for the gamma function. To arrive at \eqref{eq:m-symbol}, we use Fields's asymptotic expansion from \cite{Fields} (see also \cite{Frenzen} and 
\cite[Sec.~5.11(i)--(iii)]{DLMF}). The case $j=0$ is then immediate. Since Fields's expansion relies on a certain loop integral (see \cite[p. 43]{Fields}), we may differentiate it and thus obtain \eqref{eq:m-symbol} for all   $j>0$.
\end{proof}

The real Hilbert space
\begin{equation}\label{eq:Xalpha}
 \mathcal X_\alpha=
 \left\{f\in\PW^2:
 f\text{ is real and even on }\R,
 \ \norm{f}_\alpha^2\coloneqq 
 \int_\R(1+x^2)^{\alpha}|{f(x)}|^2\dd x<\infty
 \right\}.
\end{equation}
will be central in the proof of Lemma~\ref{prop:Gamma-correction}. A key point is that functions in $\mathcal X_\alpha$ decay faster than $E_p$. On the other hand, it is clear that $\mathcal X_{\alpha}$ is continuously embedded in $PW^p$.


\begin{lemma}\label{lem:weighted-correction-estimates}
For every integer \(j\ge0\),
\begin{equation}\label{eq:weighted-Bernstein}
 \norm{f^{(j)}}_\alpha\le C_{j,p}\norm{f}_\alpha,
 \qquad f\in\mathcal X_\alpha.
\end{equation}
In addition,
\begin{equation}\label{eq:pointwiseH}
 |f(x)|
 \le C_p(1+x^2)^{-\alpha/2}\norm{f}_\alpha.
\end{equation}
\end{lemma}

\begin{proof}
By a variant of a classical inequality of Plancherel and Pólya (see e.g. \cite[Section 2.3]{Young}),
\begin{equation} \label{eq:PP}
    \int_{-1}^1 \int_{-\infty}^\infty (1+|x+\ii y|^2)^{\alpha}|f(x+\ii y)|^2 \dd x \dd y \le C \| f\|_\alpha^2.\end{equation}
The first assertion follows by estimating $|f^{(j)}(x)|$ in terms of Cauchy estimates and then applying \eqref{eq:PP}. We get the second assertion by using subharmonicity to bound ${|f(x)|^2}$ by an area integral of $|f|^2$ in a disc of radius $1$ centered at $x$ and then invoking again \eqref{eq:PP}.
\end{proof}

We also need the following version of the elementary fact that no nonzero function can concentrate in a fixed compact interval.

\begin{lemma}\label{lem:tail-dominance}
For every \(T>1\), there is a constant \(C_{p,T}>0\) such that 
\begin{equation}\label{eq:tail-dominance}
 \norm{f}_\alpha^2
 \le C_{p,T}\int_{|x|>T}|x|^{2\alpha}|f(x)|^2\dd x,
 \qquad f\in\mathcal X_\alpha.
\end{equation}
\end{lemma}

\begin{proof}
{If the assertion failed, there would be $f_j\in\mathcal X_\alpha$
with $\|f_j\|_\alpha=1$ and weighted tail integrals tending to zero.
Pass to a weakly convergent subsequence. By \eqref{eq:PP} and local
compactness it converges locally uniformly to an entire function $f$.
The vanishing tails force $f=0$ on $(T,T+1)$, hence $f\equiv0$.
The norm on $[-T,T]$ therefore tends to zero, as does the remaining
part since $(1+x^2)^\alpha\asymp_{p,T}|x|^{2\alpha}$ for $|x|>T$.
This contradicts $\|f_j\|_\alpha=1$.}
\end{proof}

\subsection{The Fourier transform of \texorpdfstring{$x^2 |E_p(x)|^{p-2} E_p(x)$}{x squared times the Gamma power}}
\label{sec:defect}
We will now analyze the Fourier transform of the function $x^2 |E_p(x)|^{p-2} E_p(x)$. Here the point is that, by Plancherel's identity, the orthogonality relation \eqref{eq:Lportho} asserts that the Fourier transform of $x^2 |\varphi_p(x)|^{p-2}\varphi(x)$ vanishes on $(-\pi,\pi)$. With reference to our ansatz about $E_p$, we should expect the Fourier transform of   $x^2 |E_p(x)|^{p-2} E_p(x)$ to behave reasonably well on $(-\pi,\pi)$. 

We begin with a computation involving the gamma function. {Fix $\chi\in C^\infty([0,\infty))$ with $\chi=0$ on $[0,1]$ and $\chi=1$ on $[2,\infty)$.}
\begin{lemma}\label{lem:Abel-estimates}
 If \(s>-1\), then there is a function \(S\), smooth near \(\lambda=0\), such that for \(\lambda>0\),
\begin{equation}\label{eq:Abel-monomial}
 \lim_{\varepsilon\downarrow0}
 \int_0^\infty\chi(t)t^s e^{-\varepsilon t}e^{i\lambda t}\dd t
 =\Gamma(s+1)e^{i\pi(s+1)/2}\lambda^{-s-1}+S(\lambda).
\end{equation}
\end{lemma}
\begin{proof}
The result is immediate from the fact that for \(\varepsilon>0\),
\[
 \int_0^\infty t^s e^{-(\varepsilon-i\lambda)t}\dd t
 =\Gamma(s+1)(\varepsilon-i\lambda)^{-s-1}. \qedhere
\]
\end{proof}
We will also need the following computation.
\begin{lemma}\label{lem:oddn}
Let \(-2<s<0\){, $s\ne-1$,} and suppose that \(a\) in $C^1([1,\infty))$ satisfies $|a(t)|\le C_0 t^{s}$ and \(|a'(t)|\le C_1 t^{s-1}\).  Then
\begin{equation}\label{eq:symbol-bound}
 \left|\lim_{\varepsilon\downarrow0}
 \int_1^\infty a(t)e^{-\varepsilon t}e^{i\lambda t}\dd t
 \right|
 \le C\{1+\lambda^{-s-1}\}.
\end{equation}
for \(0<\lambda\le1\),
\end{lemma}

\begin{proof}
We begin by splitting the integral
at $\lambda^{-1}$. The integral from $1$ to $\lambda^{-1}$ is estimated trivially and is bounded by a constant times $1+\lambda^{-s-1}$. By integration by parts and using our assumptions on $a$, we also get
\[
 \left|\int_{\lambda^{-1}}^\infty a(t)e^{-(\varepsilon-i\lambda)t}\dd t\right|
 \le \frac{|a(\lambda^{-1})|+\int_{\lambda^{-1}}^\infty|a'(t)|\dd t}{\lambda}
 \le C \lambda^{-s-1},
\]
uniformly in $\varepsilon$.  
\end{proof}

To deal with the contribution from the sine factor in \eqref{eq:Gamma-phase} to
the Fourier transform of $x^2|E_p(x)|^{p-2} E_p(x)$, we introduce the absolutely convergent Fourier series
\begin{equation}\label{eq:cusp-series}
 F_p(\theta):=|\cos\theta|^{p-2}\cos\theta
 =\sum_{\substack{n\ge1\\ n\ \mathrm{odd}}}\beta_n(p)\cos(n\theta),
 \qquad \beta_n(p)=O_p(n^{-p}).
\end{equation}
To see the validity of \eqref{eq:cusp-series}, we observe that, for odd $n$, the beta integral gives
\[
 \beta_n(p)=
 \frac{2^{2-p}\Gamma(p)}
 {\Gamma((p+1+n)/2)\Gamma((p+1-n)/2)},
\]
{Reflection and Stirling's formula give $\beta_n(p)=O_p(n^{-p})$, and hence absolute convergence in \eqref{eq:cusp-series}.}

To simplify the writing, we will from now on use the notation
\[ \Phi_p(t)\coloneqq \frac{|t|^p}{p}, \qquad N_p(t)\coloneqq \Phi_p'(t)=|t|^{p-2}t. \]
We may now state the main result of this section.
\begin{lemma}\label{prop:Gamma-defect}
The distribution
\[
 \Delta_p:=\widehat{x^2N_p(E_p)}
\]
is represented on \((-\pi,\pi)\) by a locally integrable function, still denoted \(\Delta_p\), satisfying
\begin{equation}\label{eq:defect-bound}
 |\Delta_p(\xi)|
 \le C_p\left\{1+(\pi-\xi)^{1-\alpha}
 +(\pi+\xi)^{1-\alpha}\right\}.
\end{equation}
Moreover,
\begin{equation}\label{eq:ell0}
 \ell_0(f)=\frac1{2\pi}
 \int_{-\pi}^{\pi}\Delta_p(\xi)\widehat f(\xi)\dd\xi
\end{equation}
defines a bounded functional on \(\mathcal X_\alpha\).
\end{lemma}

\begin{proof}
We may assume that \(p\ne2\) since trivially $\Delta_2|_{(-\pi,\pi)} =0$. By Lemma~\ref{lem:Gamma-properties}, for large positive \(x\),
\begin{equation}\label{eq:defect-tail-exact}
 x^2N_p(E_p(x))
 =D_p x^\alpha m_p(x)^{p-1}
 F_p\left(\pi x-\frac\pi p\right)
\end{equation}
with $D_p$ a constant. {We insert the cutoff $\chi(x)$ in the tail integrals below; the omitted compact part has a bounded Fourier transform.}
We insert \eqref{eq:cusp-series} into \eqref{eq:defect-tail-exact}, and our task is then to compute the Fourier transform of terms of the form $x^\alpha m_p(x)^{p-1}\cos(n\pi(x-\frac1p))$. Beginning with odd $n\ge 3$, we set
\[
 I_{n,\pm,\varepsilon}(\xi)
 \coloneqq \int_0^\infty
 {\chi(x)}x^\alpha m_p(x)^{p-1}e^{-\varepsilon x}
 e^{i(\pm n\pi-\xi)x}\dd x.
\]
Let $q$ be an integer larger than $\alpha+1$. Since $|\pm n\pi - \xi|\ge(n-1)\pi$ when $|\xi|\le \pi$, $q$
integrations by parts on $[1,\infty]$ along with \eqref{eq:m-symbol} give
\[
 \sup_{0<\varepsilon\le1}\sup_{|\xi|\le\pi}
 |I_{n,\pm,\varepsilon}(\xi)|
 \le C_{q,p}.
\]
Taking into account the uniform bounds on $\beta_n(p)$ from \eqref{eq:cusp-series}, we conclude that the terms with odd $n\ge 3$ contribute at most a bounded term to $\Delta_p(\xi)$ when ${|\xi|\le\pi}$.

It remains to consider the contribution from the term with $n=1$.  Write \(m_p^{p-1}=1+s_p\).  By \eqref{eq:m-symbol},
\[
 |s_p^{(j)}(x)|\le C_{j,p}x^{-2-j}.
\]
Thus Lemma~\ref{lem:oddn} applies with $s=\alpha -2$ and so the term with $s_p$ yields a term bounded by a constant times the right-hand side of \eqref{eq:defect-bound}. Finally, to deal with the leading term when $n=1$, we see from {Lemma~\ref{lem:Abel-estimates}} with $s=\alpha$ that we may potentially have two singular terms behaving as $(\pi\mp\xi)^{-\alpha-1}$  near the respective endpoints $\xi=\pm \pi$. However, the coeffients of these terms will, according to \eqref{eq:Abel-monomial}, be 
 proportional to
\[
 \cos\left(\frac{\pi(\alpha+1)}2-\frac\pi p\right)
 =\cos\frac\pi2=0,
\]
because $\alpha=2/p$. In total, we have therefore established \eqref{eq:defect-bound}.

The final assertion, namely that $\ell_0$ defines a bounded functional on $\mathcal X_{\alpha}$, now follows from \eqref{eq:defect-bound}. For $\alpha<1$, we see this by noting that $\Delta_p$ is in $L^\infty$, and so its inverse Fourier transform is in $PW^2$. Hence the result follows from Plancherel's identity. For
$\alpha>1$, we use that $\Delta_p$ is in $L^1$ which implies that its inverse Fourier transform is bounded. Then applying Plancherel's identity again, we obtain the desired result because $\| f \|_{L^1(\R)} \le C \| f\|_{\alpha}$ when $\alpha>1$.
\end{proof}

\subsection{Identification of \texorpdfstring{$\varphi_p$}{varphi	extunderscore p} via a minimization problem in \texorpdfstring{$\mathcal X_\alpha$}{X	extunderscore alpha}}

\label{sec:Euler}

For real $a,h$, set
\begin{equation}\label{eq:Bregman-definition}
 B_p(a,h)=\Phi_p(a+h)-\Phi_p(a)-N_p(a)h.
\end{equation}
 The purpose of this section is to prove Lemma~\ref{prop:Gamma-correction} by identifying $r_p=\varphi_p-K(p)E_p$ with a scalar times the unique minimizer in \(\mathcal X_\alpha\) of the functional  
\begin{equation}\label{eq:J-functional}
 \mathcal J(H)=\ell_0(H)+\int_\R x^2B_p(E_p(x),H(x))\dd x.
\end{equation}

We will need the following bounds.
\begin{lemma}\label{lem:Bregman} We have
\begin{align}
 B_p(a,h)&\ge c_p\frac{|h|^2}{(|a|+|h|)^{2-p}}, \quad  B_p(a,h)\le C_p h^2|a|^{p-2}, &&1<p<2,
 \label{eq:Bregman-sub-lower}\\
 B_p(a,h)&\ge c_p|a|^{p-2}|h|^2, \qquad B_p(a,h)\le C_p\{|a|^{p-2}|h|^2+|h|^p\} &&p>2,
 \label{eq:Bregman-super-lower}
\end{align}
\end{lemma}
\begin{proof}
These estimates follow from the formula
\begin{equation}\label{eq:Bpint}
 B_p(a,h)=(p-1)h^2\int_0^1(1-t)|a+th|^{p-2}\dd t.
\end{equation}
To arrive at the first bound in \eqref{eq:Bregman-sub-lower} and the second bound in \eqref{eq:Bregman-super-lower}, we use that $|a+th|\le |a|+|h|$. To deal with the remaining two cases,  we write
\[B_p(a,h)=(p-1)h^2 |a|^{p-2} \int_0^1(1-t)|1+th/a|^{p-2}\dd t \]
and note that the function $x\mapsto \int_0^1(1-t)|1+tx|^{p-2}\dd t$ has a finite maximum when $1<p<2$ and a positive minimum when $p>2$.
\end{proof}

For $p>2$, we will also use the following result.

\begin{lemma}\label{lem:Gamma-dominance}
If \(p>2\), then
\begin{equation}\label{eq:Gamma-dominance}
 \norm{H}_\alpha^2
 \le C_p\int_\R x^2|E_p(x)|^{p-2}|H(x)|^2\dd x,
 \qquad H\in\mathcal X_\alpha.
\end{equation}
\end{lemma}

\begin{proof}
{Let $D_\varepsilon$ be the union of the intervals of radius
$\varepsilon<1/4$ about the zeros $\pm c_n$. Subharmonicity and
\eqref{eq:PP}, summed over unit intervals, give
\[
 \int_{D_\varepsilon}(1+x^2)^\alpha|H(x)|^2\dd x
 \le C_p\varepsilon\|H\|_\alpha^2.
\]
On $|x|>2$ outside $D_\varepsilon$, Lemma~\ref{lem:Gamma-properties}
gives $x^2|E_p(x)|^{p-2}\ge c_{p,\varepsilon}|x|^{2\alpha}$.
Apply Lemma~\ref{lem:tail-dominance} with $T=2$, split its tail into
these two sets, and choose $\varepsilon$ small enough to absorb
$C_p\varepsilon\|H\|_\alpha^2$.}
\end{proof}

Before stating the main result of this section, we note that, on $|x|\ge1$,
\begin{equation}\label{eq:Bregman-scaling}
 x^2B_p(E_p,H)=B_p(b,v),\qquad
 b\coloneqq {|x|^\alpha} E_p,\quad v\coloneqq {|x|^\alpha} H.
\end{equation}

\begin{lemma}\label{prop:Euler-solution}
For every $1<p<\infty$, $p\ne2$, the functional $\mathcal J$ has a unique minimizer $H_*$.  Moreover, for every $f$ in $\mathcal X_\alpha$,
\begin{equation}\label{eq:weak-correction}
 \ell_0(f)+\int_\R x^2
 \{N_p(E_p(x)+H_*(x))-N_p(E_p(x))\}f(x) \dd x=0.
\end{equation}
\end{lemma}

\begin{proof}
Write
\[
 \mathcal I(H)\coloneqq\int_\R x^2B_p(E_p(x),H(x))\dd x,
 \qquad \mathcal J(H)=\ell_0(H)+\mathcal I(H).
\]
We first show that $\mathcal I(H)$ is finite for every
$H$ in $\mathcal X_\alpha$.  
Suppose first that $1<p<2$. Then
\[ \int_{|x|\ge 1} B_p(b,{v}) \dd x \le C_p \int_{|x|\ge 1} |b|^{p-2} {v^2} \dd x\le C_p \| H\|_{\alpha}^2\]
by the second bound in \eqref{eq:Bregman-sub-lower}. Since the integral of $|b|^{p-2}$ over all unit intervals is uniformly bounded, this yields the result since, by \eqref{eq:PP}, we obtain an equivalent norm on $\mathcal X_\alpha$ by summing the maxima of ${v^2}$ over all intervals $[k,k+1]$, $k\in \mathbb Z$. If $p>2$, then 
the second bound in \eqref{eq:Bregman-super-lower} gives
\[
 B_p(b,v)\le C_p\{|b|^{p-2}|v|^2+|v|^p\}.
\]
{Since $b$ is bounded, $\|v\|_\infty\le C_p\|H\|_\alpha$,
and $v$ is square-integrable, it follows that
$\mathcal I(H)\le C_p(\|H\|_\alpha^2+\|H\|_\alpha^p)$.} 
Taking into account also Lemma~\ref{prop:Gamma-defect}, we conclude that
$\mathcal J$ is finite on
$\mathcal X_\alpha$.

For fixed $a$, the map $h\mapsto B_p(a,h)$ is strictly convex.  If $H_0\ne H_1$ in $\mathcal X_\alpha$, then $H_0-H_1$ is a
nonzero entire function, so its real zeros are discrete.  The strict
convexity inequality therefore holds for almost every $x$.  Integration
against $x^2\dd x$ preserves the strict inequality.  Since $\ell_0$ is
linear, we conclude that $\mathcal J$ is strictly convex.

We next verify that $\mathcal J$ is weakly lower semicontinuous. To this end,  suppose that
$H_j$ tend weakly to $H$ in $\mathcal X_\alpha$.  By
\eqref{eq:pointwiseH}, this implies that $H_j(x)\to H(x)$ for
every $x$.  Since $B_p$ is continuous and nonnegative, Fatou's lemma
gives
\[
 \mathcal I(H)\le\liminf_{j\to\infty}\mathcal I(H_j).
\]
Moreover, $\ell_0(H_j)\to\ell_0(H)$ because $\ell_0$ is bounded and
linear.  Thus
\[
 \mathcal J(H)\le\liminf_{j\to\infty}\mathcal J(H_j).
\]
This proves sequential weak lower semicontinuity.  In particular, every
sublevel set of $\mathcal J$ is norm closed.  These sets are convex, so
they are weakly closed by the Hahn--Banach separation theorem.  Thus $\mathcal J$ is weakly lower
semicontinuous.

We finally prove coercivity. For $1<p<2$, we get
\[ {\mathcal I(H)}\ge c_p \frac{\norm{H}_\alpha^2}{(1+\norm{H}_{\alpha})^{2-p}} \]
by using \eqref{eq:Bregman-scaling} and then \eqref{eq:Bregman-sub-lower} along with Lemma~\ref{lem:tail-dominance}. We conclude that coercivity holds because $\ell_0(H)\le C_p \| H\|_\alpha$. When $p>2$, appealing again to \eqref{eq:Bregman-scaling} and applying \eqref{eq:Bregman-super-lower} along with {Lemma~\ref{lem:Gamma-dominance}}, we get
\[ {\mathcal I(H)}\ge c_p \| H\|_{\alpha}^2.\]
We conclude that coercivity holds in this case as well.

Since we have shown that $\mathcal J$ is finite, coercive, strictly convex and weakly semicontinuous, we may apply the direct method from \cite[Chapter~II, Proposition~1.2, p.~35]{EkelandTemam} and conclude that there is a unique minimizer $H_\ast$ in $\mathcal X_\alpha$.

We finally compute the first variation.  Fix $H$ and $K$ in $\mathcal X_\alpha$
and, for each $x$, set
\[
 F_x(s)=x^2B_p(E_p(x),H(x)+sK(x)).
\]
The function $F_x$ is nonnegative and convex.  For $0<|t|\le1$, the
inequalities between the secant slopes of a convex function give
\[
 \left|\frac{F_x(t)-F_x(0)}{t}\right|
 \le F_x(-1)+2F_x(0)+F_x(1).
\]
The right-hand side is integrable, because $\mathcal I$ is finite at
$H-K$, $H$, and $H+K$.  Since $\Phi_p$ is continuously differentiable
for every $p>1$, the difference quotient tends pointwise to
\[
 x^2\{N_p(E_p+H)-N_p(E_p)\}K.
\]
Dominated convergence therefore gives
\[
 D\mathcal J(H)[K]=\ell_0(K)+\int_\R x^2
 \{N_p(E_p+H)-N_p(E_p)\}K\dd x.
\]
At the minimizer $H_*$ this derivative is zero for every
$K$ in $\mathcal X_\alpha$, which proves \eqref{eq:weak-correction}.

\end{proof}

\begin{proof}[Proof of Lemma~\ref{prop:Gamma-correction}{, with the identification of $K(p)$ completed below}]
{The case $p=2$ is immediate from $\varphi_2=E_2$,
with $a_2=1$ and $\widetilde r_2=0$. For $p\ne2$ and a real even
$f\in PW^p$ with $f(0)=0$, we first justify
\[
 \ell_0(f/x^2)=\int_\R N_p(E_p(x))f(x)\dd x.
\]
Choose a real even Schwartz entire function $\psi$ of type at most $\pi$
with $\psi(0)=1$, and put
$f_\varepsilon(z)=f((1-2\varepsilon)z)\psi(\varepsilon z)$,
$0<\varepsilon<1/2$. The identity holds for $f_\varepsilon$ because
$\widehat{f_\varepsilon/x^2}$ is smooth and supported strictly inside the band.
Moreover, $f_\varepsilon\to f$ in $L^p$ and
$f_\varepsilon/x^2\to f/x^2$ in $\mathcal X_\alpha$:
near zero use the double zero and local uniform convergence;
on the tails use $PW^p\subset PW^2$ for $p\le2$, and H\"older's
inequality with the weight $|x|^{2\alpha-4}$ for $p>2$.
Pass to the limit using boundedness of $\ell_0$ and H\"older's inequality.
}
It remains only to verify that $E_p(0)+H_\ast(0)\neq 0$. Indeed, when $f$ is a {real} even function in $PW^p$ vanishing at $0$, $f(x)/x^2$ is in $\mathcal{X}_\alpha$, and so \eqref{eq:weak-correction} shows that 
\begin{equation} \label{eq:RK} \int_\R 
 N_p(E_p(x)+H_*(x))f(x) \dd x=0.\end{equation} 
 Hence \eqref{eq:Lportho} holds for 
 \[\varphi(x)\coloneqq \frac{E_p(x)+H_*(x)}{E_p(0)+H_*(0)},\]
 provided that $E_p(0)+H_\ast(0)\neq 0${.} However, if $E_p(0)+H_\ast(0)=0$, then \eqref{eq:RK} would apply to 
 $f=E_p + H_\ast$, and we would get $E_p(x)+H_*(x)\equiv 0$. But this is impossible since $E_p$ does not belong to $\mathcal X_\alpha$.
 {Thus, with $a_p=(1+H_*(0))^{-1}\ne0$,
$\widetilde r_p:=\varphi_p-a_pE_p=a_pH_*\in\mathcal X_\alpha$.
After proving indexed localization we will identify $a_p=K(p)$.}
\end{proof}
As noted in the introduction, the following consequence of {the preceding construction} will be the starting point for our bootstrap argument leading to Theorem~\ref{thm:main}. 
\begin{lemma} \label{cor:l2-localization} The sequence $(\tau_n-c_n)$ is in $\ell^2$. \end{lemma}
\begin{proof}
We set $J_n\coloneqq [c_n-\frac14,c_n+\frac14]$ and show first that $\varphi_p$ has exactly one zero in $J_n$ for all sufficiently large $n$. To this end, we begin by noting that the {absolute value of the} derivative of $x^\alpha E_p(x)$ is bounded below on \(J_n\). On the other hand, since {$\widetilde r_p$ from the preceding construction} is in $\mathcal{X}_\alpha$, both ${x^\alpha\widetilde r_p(x)}$ and its derivative tend to $0$ when $x\to \infty$. This implies that, when $n$ is large enough, say $n\ge N$,
$x^\alpha \varphi(x)$ is monotone and changes sign on $J_n$. Hence it has exactly one zero there which we will temporarily call $\rho_n$.

The mean-value theorem and the lower bound for the {absolute value of the} derivative of $x^\alpha \varphi_p(x)$ give
\begin{equation}\label{eq:displacement}
 |\rho_n-c_n|\le C_pc_n^{\alpha}\, |{\widetilde r_p(c_n)}|.
\end{equation}
Then applying \eqref{eq:PP} to ${\widetilde r_p}$, we find that the sequence $(\rho_n-c_n)_{n\ge N}$ is square-summable.

Since $\varphi_p$ is an even function, 
\[ \varphi_p(x)=g(x)\prod_{n\ge N} \left(1-\frac{x^2}{\rho_n^2}\right),\]
{where $g$ is a polynomial: for large $x$ outside the intervals $J_n$,
$|x^\alpha E_p(x)|$ is bounded below and
$x^\alpha\widetilde r_p(x)\to0$, so there are no additional large real zeros.
The assertion follows from the paired product \eqref{eq:phip}.
It remains to show that its degree is $2(N-1)$.} Straightforward estimates, using that
$(\rho_n-c_n)_{n\ge N}$ is in $\ell^2$, show that
\[ \prod_{n\ge N} \frac{\left(1-\frac{(c_k+\frac12)^2}{c_n^2}\right)}
{\left(1-\frac{(c_k+\frac12)^2}{\rho_n^2}\right)} \asymp 1, \quad k\ge N. \]
Since we also have $|\varphi_p(c_k+\frac12)|\asymp |E_p(c_k+\frac12)|$, we find that
\[ c_k^{2(N-1)} {\left|\prod_{n\ge N}
\left(1-\frac{(c_k+\frac12)^2}{\rho_n^2}\right)\right|} \asymp c_k^{-\alpha}. \]
{It follows that $|g(c_k+\frac12)|\asymp c_k^{2(N-1)}$,
so $\deg g=2(N-1)$ and $\rho_n=\tau_n$ for $n\ge N$.
The product $K(p)=\prod_n(c_n/\tau_n)^2$ now converges.
At $x_k=c_k+\frac12$, the paired-product quotient satisfies
$\varphi_p(x_k)/E_p(x_k)\to K(p)$: split its logarithm into a
finite head and an $\ell^2$ tail, using the uniformly bounded
$\ell^2$ norm of $((c_n-x_k)^{-1},(c_n+x_k)^{-1})_n$.
On the other hand, $\widetilde r_p(x_k)/E_p(x_k)\to0$, so $a_p=K(p)$.
This also completes the proof of Lemma~\ref{prop:Gamma-correction}.}
\end{proof}

\section{A bootstrap argument}
\label{sec:Gamma}

\subsection{The operator \texorpdfstring{$\mathscr{Z}$}{Z} and its Jacobian} Adopting a terminology used in \cite{BORS}, we say that an increasing sequence of positive numbers  $\mathbf t=(t_n)$ is admissible if $\mathbf t - \mathbf c$ is in $\ell^2$, where $\mathbf c \coloneqq (c_n)$ are the positive zeros of $E_p$. We proved in the previous section that $\bm{\tau} \coloneqq (\tau_n)$ is an admissible sequence, and our goal is now to strengthen this to obtain Theorem~\ref{thm:main}. 

We set
\begin{equation}\label{eq:paired-product}
 \phi_{\mathbf t}(z)\coloneqq \prod_{n\ge1}\left(1-\frac{z^2}{t_n^2}\right)
\end{equation}
which is an entire function of exponential type $\pi$ with the same asymptotic behavior as $E_p$.   Put
\begin{equation}\label{eq:R-def}
 R_k(\mathbf t)=\int_\R|\phi_{\mathbf t}(x)|^p
 \frac{x^2}{t_k^2-x^2}\dd x.
\end{equation}
Since the functions $\phi_{\bm{t}}(x)x^2/(x^2-t_n^2)$ constitute a complete sequence in the subspace of $PW^p$ consisting of even functions vanishing at $0$ and $\bm{\tau}$ is an admissible sequence, we have, by \eqref{eq:Lportho}, that  $\bm{t}=\bm{\tau}$ if and only if
\[ R_k(\bm{t})=0, \qquad k\ge 1.\]
This, combined with an even power series expansion of $R_k(\bm c)$ in terms of $c_n^{-1}$, will be the starting point for our bootstrap argument.
We will see that a similar expansion of a finite perturbation of $\bm c$, along with the facts that $\bm{\tau}-\bm{c}$ is square-summable and $R_k(\bm{\tau})=0$, forces $\bm{\tau}-\bm{c}$, by an induction argument, to enjoy itself a power series expansion in terms of $c_n^{-1}$. 

We will need some precise quantitative information about the operator $\bm t \mapsto \mathscr{Z}\bm{t}$, where $\mathscr{Z}\bm{t}=(z_k)$ is a sequence whose entries are
\begin{equation} \label{eq:zkdef} z_k \coloneqq t_k \kappa(\bm{t}) R_k(\bm{t}), \qquad \kappa(\bm{t})\coloneqq \prod_{n\ge 1} \left(\frac{t_n}{c_n}\right)^{2p}.\end{equation}
    Note that the infinite product $\kappa(\bm{t})$ converges on the assumption that $\bm t$ is an admissible sequence. Since $(\mathscr Z \bm{t})_k$ is just a normalization of $R_k(\bm{t})$, we have $\mathscr Z\bm\tau=0$ if and only if $\mathscr{Z}\bm{t}=0$.
Our choice of normalization of $R_k(\bm t)$ is justified by the following lemma, which shows that the Jacobian $D\mathscr Z(\bm{t})$ of $\mathscr{Z}\bm{t}$ has a particularly simple form when $\bm{t}=\bm{\tau}$.

\begin{lemma}\label{lem:normalized-zero-algebra}
Let $\mathbf t$ be an admissible sequence. Then
\begin{equation} \label{eq:Jacobian}
\partial_{t_j} z_k 
 =\begin{cases} 2p\,\frac{t_jz_k-t_kz_j}{t_j^2-t_k^2}, & j\neq k ;\\
 \kappa(\bm{t})\big((1+2p)R_k(\bm{t})+t_k \partial_{t_k} R_k(\bm{t})\big), & j=k.\end{cases} 
\end{equation}
In particular,
\begin{equation}\label{eq:normalized-scalar-hessian}
 D\mathscr Z(\boldsymbol\tau)=\lambda_p I,
 \qquad
 \lambda_p\coloneqq  \frac{\kappa(\bm{\tau})}{\mathscr C_p}.
\end{equation}
\end{lemma}

\begin{proof}
We begin by differentiating $R_k=R_k(\bm{t})$. To this end, we 
note to begin with that
\begin{equation}\label{eq:help}
 \partial_{t_k}|\phi_{\mathbf t}(x)|^p
 =\frac{2p}{t_k}|\phi_{\mathbf t}(x)|^p
 \frac{x^2}{t_k^2-x^2},
\end{equation}
When $j\ne k$, we  use that
\[
 \frac{x^4}{(t_j^2-x^2)(t_k^2-x^2)}
 =\frac{t_k^2}{t_j^2-t_k^2}\frac{x^2}{t_k^2-x^2}
 -\frac{t_j^2}{t_j^2-t_k^2}\frac{x^2}{t_j^2-x^2}
\]
and find that 
\begin{equation}\label{eq:offdiag-Jacobian}
 \partial_{t_j} R_k
 =\frac{2p}{t_j}\frac{t_k^2R_k-t_j^2R_j}{t_j^2-t_k^2}.
\end{equation}
Noting that 
\[ \partial_{t_j} \kappa(\bm{t})=\frac{2p}{t_j} {\kappa(\bm t)},\]
we therefore get
\[ \partial_{t_j} z_k = \kappa(\bm{t})t_k\big(\frac{2p}{t_j}R_k+\partial_{t_j} R_k\big)=2p\,\frac{t_jz_k-t_kz_j}{t_j^2-t_k^2},\]
as desired. The case $j=k$ in \eqref{eq:Jacobian} follows immediately by Leibniz's rule. 

Since $R_k(\bm\tau)=0$, $D\mathscr Z(\bm{\tau})$ is a diagonal matrix. By \eqref{eq:Jacobian}, the $k$th diagonal entry is
\[  \kappa(\bm{\tau})t_k \partial_{t_k} R_k(\bm{\tau}).  \]
Hence it remains to show that
\begin{equation} \label{eq:ident1} \tau_k \partial_{\tau_k} R_k(\bm{\tau})=\| \varphi_p\|_p^p.\end{equation}
To prove this, we start by noting that 
\[ \partial_{\tau_k} R_k(\bm{\tau}) = \int_{\R} \frac{| \varphi_p(x)|^px^2}{(\tau_k^2-x^2)^2}\Big(\frac{2p}{\tau_k}x^2-2\tau_k\Big) \dd x\]
by Leibniz's rule along with \eqref{eq:help}. It follows that
\begin{equation} \label{eq:almost} \tau_k \partial_{t_k} R_k(\bm{\tau})=
2\tau_k^2(p-1) \int_{\R} \frac{| \varphi_p(x)|^px^2}{(\tau_k^2-x^2)^2} \dd x
= (p-1) \int_{\R} \frac{| \varphi_p(x)|^px^2}{(\tau_k-x)^2} \dd x,
\end{equation}
where in the last step we used \eqref{eq:Lportho}. 
To compute the latter integral, we write
\begin{equation} \label{eq:tran} \int_{\R} \frac{| \varphi_p(x)|^px^2}{(\tau_k-x)^2} \dd x=\tau_k^2\int_{\R} \frac{| \varphi_p(x)|^p}{(\tau_k-x)^2} \dd x-\| \varphi_p\|_p^p\end{equation}
by \cite[Theorem 3.6]{BCOS}. Since $p>1$, we may use integration by parts to get
\[ \int_{\R}\frac{| \varphi_p(x)|^p}{(\tau_k-x)^2} \dd x= -p \int_{\R}\frac{| \varphi_{p}(x)|^{p-2}\varphi_p(x) \varphi_p'(x)}{\tau_k-x} \dd x.\]
Since the function $\frac{\varphi_p(x)}{(x-\tau_k)^2}-\frac{\varphi'(x)}{x-\tau_k}$ belongs to $PW^p$, another application of \cite[Theorem 3.6]{BCOS} yields
\[ \tau_k^2\int_{\R}\frac{| \varphi_p(x)|^p}{(\tau_k-x)^2} \dd x=\frac{p}{p-1} \|\varphi_p\|_p^p. \]
Plugging this into \eqref{eq:tran} and returning to \eqref{eq:almost}, we arrive at the desired conclusion that $D\mathscr{Z}=\lambda_p I$.
\end{proof}

\subsection{Estimates of \texorpdfstring{$\phi_{\mathbf{t}}$}{phi t}}

The following estimate will be used below.

\begin{lemma}\label{lem:local-product-comparison}
Let $\bm t$ be an admissible sequence. Then
\begin{equation}\label{eq:local-product-envelope}
 |\phi_{\bm{t}}(x)|\asymp 
 (1+|x|)^{-\alpha}
 \min\{1,\dist(x,\{\pm t_n:n\geq1\})\}.
\end{equation}
Moreover, 
\begin{equation}\label{eq:local-product-derivative}
 |\phi_{\bm{t}}'(t_k)|\asymp_p t_k^{-\alpha}
\end{equation}
and \begin{equation} \label{eq:diag}
{|\partial_{t_k} z_k|} \le C(\bm{t})
\end{equation}
uniformly in $k$. Here $C(\bm{t})$ and the implicit constants in \eqref{eq:local-product-envelope} and \eqref{eq:local-product-derivative} can be chosen to depend only on the $\ell^2$ norm of $\bm c - \bm t$ and the separation of $\bm t${, including its distance from zero}.
\end{lemma}

\begin{proof}
Write $a_n\coloneqq t_n-c_n$ so that
\begin{equation}\label{eq:normalized-factorization}
 \phi_{\bm{t}}(z)={\kappa(\bm t)^{-1/p}}E_p(z)
 \prod_{n\geq1}
 \left(1+\frac{a_n}{c_n-z}\right)
 \left(1+\frac{a_n}{c_n+z}\right).
\end{equation}
Hence
\[ |\phi_{\bm t}(x)|\asymp \frac{\dist(x,\{\pm t_n:n\geq1\})}{\dist(x,\{\pm c_n:n\geq1\})}|E_p(x)|,\]
where the implicit constants can taken to depend on the $\ell^2$-norm of $\bm c - \bm t$ and the separation of $\bm t$. Both \eqref{eq:local-product-envelope} and \eqref{eq:local-product-derivative} follow from this estimate. Starting from \eqref{eq:Jacobian}, we find that
\[ \partial_{t_k} z_k={\kappa(\bm t)}{\Bigg[}\int_{\R} \frac{|\phi_{\bm t}(x)|^p x^2}{t_k^2-x^2} \dd x
+2(p-1) t_k^2 \int_{\R} \frac{|\phi_{\bm t}(x)|^px^2}{(t_k^2-x^2)^2} \dd x{\Bigg]}.\]
Applying \eqref{eq:local-product-envelope}, we get \eqref{eq:diag}, with $C$ depending on the $\ell^2$-norm of $\bm c - \bm t$ and the separation of $\bm t$.
\end{proof}
We may now establish \eqref{eq:Cpformula}. We first note that, by \eqref{eq:ident1} and \eqref{eq:almost}, we have
\begin{equation} \label{eq:Cpcomp} \frac{1}{\mathscr C_p}= (p-1) \int_{\R} \frac{| \varphi_p(x)|^px^2}{(\tau_k-x)^2} \dd x .\end{equation} 
{Split \eqref{eq:normalized-factorization} into a finite head
and an $\ell^2$ tail, removing the nearest factors before estimating
the tail by Cauchy--Schwarz. Since $\tau_k-c_k\to0$,
Lemma~\ref{lem:Gamma-properties} gives, locally uniformly in $u$,}
\[ {|\varphi_p(\tau_k+u)|^p(\tau_k+u)^2\longrightarrow} \left[
 \frac{\Gamma\!\left(\frac12+\frac1p\right)^2}{\pi}
 \right]^p
 \prod_{n=1}^{\infty}
 \left(
 \frac{n+\frac1p-\frac12}{\tau_n(p)}
 \right)^{2p} {|\sin\pi u|^p}\]
{as $k\to\infty$. After division by $u^2$,
\eqref{eq:local-product-envelope} bounds the translated integrand by
$C_p\min(1,|u|^p)/u^2$, which is integrable for $p>1$.
Dominated convergence in \eqref{eq:Cpcomp} therefore gives}
\[ \frac{1}{\mathscr C_p}= (p-1)\left[
 \frac{\Gamma\!\left(\frac12+\frac1p\right)^2}{\pi}
 \right]^p
 \prod_{n=1}^{\infty}
 \left(
 \frac{n+\frac1p-\frac12}{\tau_n(p)}
 \right)^{2p} \int_{\R} \frac{|\sin \pi x|^p}{x^2} \dd x\]
The identity
\(\sum_{k\in\mathbb Z}(u+k)^{-2}=\pi^2/\sin^2(\pi u)\) gives
\[
 \int_{\R} \frac{|\sin \pi x|^p}{x^2} \dd x
 =\int_0^1\sin^p(\pi u)
   \sum_{k\in\mathbb Z}\frac{1}{(u+k)^2}\,\dd u
 =\pi^{3/2}
   \frac{\Gamma((p-1)/2)}{\Gamma(p/2)},
\]
and so we arrive at \eqref{eq:Cpformula}.

\subsection{Local \texorpdfstring{$\ell^2$}{l2} boundedness of \texorpdfstring{$\mathscr{Z}$}{Z}}

The next lemma gives the local continuity of $\mathscr Z$ that is
needed below. We will use the notation
\[ \mathcal B(\bm \tau;\rho)\coloneqq \{ \bm t: \norm{\bm \tau-\bm t}_2{<} \rho\}.\]
We see that for sufficiently small $\rho$, $\mathcal{B}(\bm \tau;\rho)$ will consist of admissible sequences.
\begin{lemma}\label{lem:residual-lipschitz}
There exists a $0<\rho_0\le \tfrac18$ and positive constant $C$ such that
\begin{equation}\label{eq:residual-two-point-Lipschitz}
 \|\mathscr Z(\mathbf t)-\mathscr Z(\mathbf s)\|_2
 \leq C \|\mathbf t-\mathbf s\|_2,
 \qquad \mathbf t,\mathbf s\in\mathcal B(\bm \tau;\rho_0).
\end{equation}
\end{lemma}

\begin{proof}
{Choose $0<\rho\le1/8$ so that $\mathcal B(\bm\tau;\rho)$
consists of sequences uniformly separated from each other and zero. Let $\bm h$ have
finite support with $\bm\tau+\bm h\in\mathcal B(\bm\tau;\rho)$.} Set $\mathbf t(s)=\boldsymbol\tau+s\bm h$, and choose $m$ such that the support of $h$ is contained in $\{1,\ldots, m\}$.  Let $P_m$ denote projection onto the first $m$
coordinates in $\ell^2$. Then
\begin{equation} \label{eq:FTA}  P_m \mathscr Z (\bm \tau+\bm h)= \int_0^1 \frac{\dd}{\dd s }
P_m \mathscr Z (\bm t(s))\dd s\end{equation}
since $\mathscr Z(\bm \tau)=0$. The derivative in the integral is
\[ \big(\frac{\dd}{\dd s }
P_m \mathscr Z (\bm t(s))\big)_k=\sum_{j=1}^m \partial_{t_j}z_k h_j. \]
By \eqref{eq:Jacobian}, we have, for $j\neq k$,
\[ \partial_{t_j}z_k=p\left(\Big(\frac{1}{t_j-t_k}+\frac{1}{t_j+t_k}\Big)z_k 
{-}\Big(\frac{1}{t_j-t_k}-\frac{1}{t_j+t_k}\Big)z_j\right).\]
Applying Montgomery and Vaughan's generalization of Hilbert's inequality \cite[Theorem 2]{MontgomeryVaughan} and taking into account \eqref{eq:diag}, we therefore find that
\[
 \left\|\frac{\dd}{\dd s}P_m\mathscr Z(\mathbf t(s))\right\|_2
 \leq C_p\|h\|_2
 \left\{1+\|P_m\mathscr Z(\mathbf t(s))\|_2\right\}.
\]
{This holds uniformly in $s$ because $\bm t(s)\in\mathcal B(\bm\tau;\rho)$.
Choose $\rho_0=\min(\rho,1/(4C_p))$ and $\|h\|_2<\rho_0$.
Apply \eqref{eq:FTA} up to each $s\in[0,1]$, take the supremum
over $s$, and absorb $C_p\|h\|_2\sup_s\|P_m\mathscr Z(\bm t(s))\|_2$.
Then we get}
\begin{equation}\label{eq:finite-direction-residual-bound}
 \|\mathscr Z(\boldsymbol\tau+\bm h)\|_2\leq {\frac{4C_p}{3}} \|h\|_2,
\end{equation}
on letting $m$ tend to $\infty$. 

For finitely supported $\bm h$ and $\bm g$, we apply the
same argument to the segment joining $\boldsymbol\tau+\bm h$ to
$\boldsymbol\tau+\bm g$ and get
\begin{equation}\label{eq:finite-two-point-residual-bound}
 \|\mathscr Z(\boldsymbol\tau+\bm h)-\mathscr Z(\boldsymbol\tau+\bm g)\|_2
 \leq C_p\|\bm h-\bm g\|_2.
\end{equation}
For an arbitrary $\bm h=(h_n)$ in $\ell^2$, let $\bm h_k$ be the truncated sequence whose entries are $h_1,\ldots, h_k, 0, 0, \cdots$. The bound in \eqref{eq:finite-two-point-residual-bound} shows that ${\mathscr Z(\boldsymbol\tau+\bm h_k)}$ is a Cauchy sequence in $\ell^2$. The coordinate-wise limit of this Cauchy sequence is $\mathscr Z(\boldsymbol\tau+\bm h)$ according to Lemma~\ref{lem:local-product-comparison}, and so  we obtain \eqref{eq:residual-two-point-Lipschitz}.
\end{proof}


\subsection{The diagonal of the Jacobian of \texorpdfstring{$\mathscr{Z}$}{Z}} The following identity gives the diagonal of the matrix $(\partial_{t_j} z_k)$ from its
off-diagonal entries. Note that the identity reduces to \eqref{eq:normalized-scalar-hessian} when $\bm t =\bm \tau$.

\begin{lemma}\label{lem:squared-zero-shift}
Let $\bm t$ be a sequence in $\mathcal{B}(\bm \tau,\rho_0)$. Then  
\begin{equation}\label{eq:diagonal-residual-formula}
 \partial_{t_k} z_k =\lambda_p+2\frac{z_k}{t_k}+2p\left\{
 \sum_{j\geq1}\frac{z_j}{t_j}-t_k^2\sum_{j\ne k}
 \frac{\frac{z_k}{t_k}-\frac{z_j}{t_j}}{t_j^2-t_k^2}\right\}.
\end{equation}
\end{lemma}

\begin{proof}
Let $A$ be the matrix with entries $A_{k,j}\coloneqq \partial_{t_j} z_k$, and let $A^\circ$ be the off-diagonal part of $A$. Set $\bm v\coloneqq (1/t_j)$. We will identify $A \bm v$ so that
\[ \partial_{t_k} z_k = A_{kk} = t_k (A \bm v)_k-t_k(A^\circ \bm v)_k.\]
In view of \eqref{eq:Jacobian},
\[ (A^\circ \bm v)_k =2p \sum_{j\neq k} \frac{t_jz_k-t_kz_j}{t_j^2-t_k^2} \frac{1}{t_j} ,\]
and so our task is to show that
\begin{equation}\label{eq:task} t_k (A \bm v)_k =\lambda_p+2\frac{z_k}{t_k}+2p
 \sum_{j\geq1}\frac{z_j}{t_j} .\end{equation}

We use the sequence  $v_j=t_j^{-1}$ because we rely on a computation involving  the deformation
\[
 t_j(\varepsilon)^2=t_j^2+\varepsilon,
 \qquad t_j'(0)=\frac1{2t_j}.
\]
Indeed, we may write 
\begin{equation} \label{eq:Jacobianv} (A\bm v)_k
 =2 \left.\frac{\dd}{\dd\varepsilon}
       \bigl(t_k(\varepsilon)r_k(\varepsilon)\bigr)
       \right|_{\varepsilon=0},
\end{equation}
where 
\[
 r_k(\varepsilon)\coloneqq \int_{\R}|F_{\bm t(\varepsilon)}(x)|^p
                         \frac{x^2}{t_k^2+\varepsilon-x^2}\dd x, \qquad F_{\bm t}(x)\coloneqq \kappa(\bm t)^{\frac{1}{p}} \phi_{\bm t}(x).
\]
To compute the right-hand side of \eqref{eq:Jacobianv}, we put $G(x)\coloneqq |F_{\bm t}(x)|^p$ and note that $F_{\bm t(\varepsilon)}(x)
 =F_{\bm t}\bigl((x^2-\varepsilon)^{1/2}\bigr)$
Differentiation and
integration by parts give
\begin{align*}
 r_k'(0)
 &=-\frac12\int_{\R}\frac{xG'(x)}{t_k^2-x^2}\dd x
   -\int_{\R}\frac{x^2G(x)}{(t_k^2-x^2)^2}\dd x\\
 &=\frac12\int_{\R}\frac{G(x)}{t_k^2-x^2}\dd x
  =\frac{\|F_{\bm t}\|_p^p+z_k/t_k}{2t_k^2}.
\end{align*}
Taking into account \eqref{eq:Jacobianv}, we therefore have 
\begin{equation}\label{eq:stop}
 t_k (A\bm v)_k
 = \|F_{\bm t}\|_p^p+2 \frac{z_k}{t_k}.
\end{equation}

We now consider the function
\[ R(\bm t)\coloneqq \| F_{\bm t} \|_p^p - 2p \sum_{j\ge 1} \frac{z_j}{t_j}\]
for which we have $R(\bm \tau)=\lambda_p$. Therefore it suffices to show that $R(\bm t)$ is constant for $\bm t$ in $\mathcal{B}(\bm \tau,\rho_0)$. {By continuity it suffices to check constancy on finite-coordinate
deformations, so it is enough that all partial derivatives vanish.
The differentiated sums below are justified by the local operator bounds
and $(1/t_j)\in\ell^2$; the matrix $A$ is symmetric by \eqref{eq:Jacobian}.} We find that
\begin{equation}\label{eq:incred}
 \partial_{t_k} \sum_{j\ge 1} \frac{z_j}{t_j}
 =\sum_{j\geq1}\frac{A_{j k}}{t_j}
      -\frac{z_k}{t_k^2}\\
 =(A\bm v)_k-\frac{z_k}{t_k^2}=\frac{\| F_{\bm t}\|_p^p}{t_k}+\frac{z_k}{t_k^2}, \end{equation}
where, in the last step, we used \eqref{eq:stop}. On the other hand, differentiation of the normalized product gives
\[
 \partial_{t_k}|F_{\bm t}(x)|^p
 =\frac{2p\,t_k}{t_k^2-x^2}|F_{\bm t}(x)|^p, 
\]
and so
\[
 \partial_{t_k} \|F_{\bm t} \|_p^p
 = 2p \Big( \frac{\| F_{\bm t}\|_p^p}{t_k} + \frac{z_k}{t_k^2}\Big).
\]
Combining this with \eqref{eq:incred}, we find that $R(\bm t)\equiv \lambda_p$ for $\bm t$ in $\mathcal{B}(\bm \tau,\rho_0)$, and so we conclude that \eqref{eq:task} holds.
\end{proof} 

\subsection{A local inverse of \texorpdfstring{$\mathscr Z$}{Z}}

We next establish the local invertibility of $\mathscr Z$. Here  $X:=\ell^\infty_1$ denotes the space of sequences $\bm a =(a_j)$ satisfying $\sup_j |a_j|j<\infty$. 

\begin{lemma}\label{prop:local-residual-inverse}
There exist positive numbers $\rho_0, \sigma_0$ such that  
\begin{equation}\label{eq:local-residual-diffeomorphism}
 \mathscr Z: \mathcal{B}(\bm \tau,\rho_0)\longrightarrow V:= \mathscr Z(\mathcal{B}(\bm \tau, \rho_0))
\end{equation}
 is a diffeomorphism  and  $B(\bm 0, \sigma_0) \subset V$. Moreover, 
if $\mathscr Z \bm t$ is in $\ell_1^\infty$, then $D\mathscr Z(\mathbf t)^{-1}$ maps $\ell^\infty_1$ into $\ell^\infty_1$, with $\|D\mathscr Z(\mathbf t)^{-1}\|_{\ell^\infty_1\to\ell^\infty_1}$ uniformly bounded on {$\bm t \in \mathcal{B}(\bm \tau, \rho_0)$ satisfying $\|\mathscr Z\bm t\|_{\ell^\infty_1}\le B$, for each fixed $B$}.


\end{lemma}

\begin{proof}
Lemma~\ref{lem:squared-zero-shift}
shows that the diagonal of $A=A(\bm t) \coloneqq D \mathscr Z \bm t$ is continuous in
$\ell^\infty$.  At $\boldsymbol\tau$, it equals
$\lambda_pI$ by Lemma~\ref{lem:normalized-zero-algebra}. We also have $A(\bm \tau)^\circ \bm \delta =0$, so we may write
\[ A(\bm t)^\circ\bm \delta=(A(\bm t)^\circ-A(\bm \tau)^\circ)\bm \delta .\]  A direct computation shows that $\|A(\bm t)^\circ-A(\bm \tau)^\circ\|_{2\to 2}\to 0$ when $ \bm t - \bm \tau  \to 0$ in $\ell^2$.  In total, we therefore find that
\begin{equation}\label{eq:differential-close-scalar}
 \|A (\mathbf t)-\lambda_pI\|_{2\to2}\longrightarrow0
 \qquad(\mathbf t -\bm \tau\to \bm 0 \text{ in }\ell^2).
\end{equation}
This implies that we may find $\rho$ such that
 \[ \|\lambda_p^{-1} A(\mathbf t)-I\|_{2\to2} \le \frac12, \qquad \bm t \in \mathcal B(\bm \tau, \rho).\]
We may therefore express the inverse of $A$ on $\ell^2$ as a convergent Neumann series,
\[ A^{-1} \bm \nu =\frac{1}{\lambda_p} \sum_{n=0}^\infty (I-\lambda_p^{-1} A)^n \bm \nu  \]
for every $\bm t$ in $\mathcal B(\bm \tau, \rho)$.
{Finite-coordinate differentiation and truncation, using the
locally continuous operator bounds above, show that $\mathscr Z$ is $C^1$.}
Since the Fr\'{e}chet derivative of $\bm t \to \mathscr Z \bm t$ is a bounded linear isomorphism in $\mathcal B(\bm \tau, \rho)$, the inverse function theorem \cite[p. 149]{Deimling} guarantees that there exists $0<\rho_0 \le  \rho$ such that  $\bm t \to \mathscr Z \bm t$ is invertible from $\mathcal B(\bm \tau, \rho_0)$ onto $V= \mathscr Z( B(\bm \tau, \rho_0))$. 

It remains to prove the assertion about how $D\mathscr Z(\mathbf t)^{-1}$ acts on $\ell^\infty_1$ when $\mathscr Z\bm t$ belongs to that space. 
We make a block decomposition of $A$:
\[
A=\begin{pmatrix} Q & R \\ C & T \end{pmatrix},
\]
where $Q$ is the $N\times N$ matrix $(A_{kj})_{k,j=1}^N$
and $N$ is a positive integer to be suitably chosen.
The estimate $\|A-\lambda_pI\|_{2\to2}\le\lambda_p/2$
also holds for $Q$, so $Q$ is invertible and
$\|Q^{-1}\|_{2\to2}\le2/\lambda_p$.
Since both $A$ and $Q$ are invertible, their Schur complement $T-CQ^{-1}R$  is
invertible.
Setting $W\coloneqq (T-C Q^{-1} R)^{-1}$ and following
\cite[pp.~18--19]{HornJohnson}, we may write the blockwise inversion of $A$ as
\[
 A^{-1}=\begin{pmatrix}
 Q^{-1}+Q^{-1}RWCQ^{-1} & -Q^{-1}RW \\
 -WCQ^{-1} & W
 \end{pmatrix}.
\]
Recall that  
$\|v\|_X=\sup_k k|v_k| <\infty$. Fix $B$ with $\|z\|_X\le B$, and write
$X_N=\{(v_k)_{k>N}:\sup_{k>N}k|v_k|<\infty\}$.
For $j\ne k$, \eqref{eq:Jacobian} and the uniform separation give
\[
 \frac{k}{j}|A_{kj}|
 \le C_pB\frac{1+(k/j)^2}{|j-k|(j+k)}.
\]
On summing over $j>N$, the ranges $j<k/2$,
$k/2\le j\le2k$ (with $j\ne k$), and $j>2k$ contribute at most
$C_pB/N$, $C_pB\log(2+k)/k$, and $C_pB/k$, respectively.
Consequently
\begin{equation}\label{eq:weighted-tail-row-estimate}
 \|T^\circ\|_{X_N\to X_N}
 \le\sup_{k>N}\sum_{\substack{j>N\\j\ne k}}
                 \frac{k}{j}|A_{kj}|
 \le C_pB\frac{1+\log N}{N}.
\end{equation}
We must also estimate  $CQ^{-1}R$.
For $j\le N<k$, the same entrywise bound gives
\[
 k|A_{kj}|\le C_pB\frac{j+k^2/j}{(k-j)(k+j)}
 \le C_pB\left(\frac1j+\frac1{k-j}\right).
\]
The sum of the squares of the last expression over $j\le N$ is
bounded by $C_pB^2$, uniformly in $k>N$ and $N$.
Cauchy--Schwarz therefore gives
$\|C\|_{\ell^2_N\to X_N}\le C_pB$.
On the other hand, $R$ is a compression of $A-\lambda_pI$, and
\[
 \|v\|_2^2\le\|v\|_{X_N}^2\sum_{k>N}k^{-2}
 \le N^{-1}\|v\|_{X_N}^2.
\]
Consequently
\[
 \|R\|_{X_N\to\ell^2_N}\le\frac{\lambda_p}{2\sqrt N},
 \qquad
 \|CQ^{-1}R\|_{X_N\to X_N}\le\frac{C_pB}{\sqrt N}.
\]
Let $\Delta$ be the diagonal of $T$ and put
$E=T^\circ-CQ^{-1}R$. Since $|A_{kk}|\ge\lambda_p/2$,
\eqref{eq:weighted-tail-row-estimate} implies
\[
 \|\Delta^{-1}E\|_{X_N\to X_N}
 \le\frac{C_pB}{\lambda_p}
       \left(\frac{1+\log N}{N}+\frac1{\sqrt N}\right).
\]
Choose $N=N(p,B)$ so that this is at most $1/2$.
The Neumann series
\[
 (T-CQ^{-1}R)^{-1}
 =\sum_{m\ge0}(-\Delta^{-1}E)^m\Delta^{-1}
 \quad\text{on }X_N
\]
then converges, with inverse norm at most $4/\lambda_p$.
The displayed blockwise formula now shows that $A^{-1}$ maps
$X$ into itself: $C$ maps the finite head into $X_N$, $R$ maps
$X_N$ into the finite head, and $W$ is bounded on $X_N$.
The norm comparisons on the head cost only constants depending on
$N$; since $N$ depends only on $p,B$, the resulting bound is
$\|A^{-1}g\|_X\le C_{p,B}\|g\|_X$, as required.
\end{proof}


\subsection{Starting point for the induction}

For a fixed integer $K\geq0$, let
\begin{equation}\label{eq:finite-comparison-sequence}
 \mathbf a^{(K)}=(\tau_1,\ldots,\tau_K,c_{K+1},c_{K+2},\ldots),
 \qquad f^{(K)}=\mathscr Z(\mathbf a^{(K)}).
\end{equation}
The next lemma gives the expansion that starts the induction.

\begin{lemma}\label{lem:finite-comparison-residual}
For every fixed $K$ and every $L\geq1$, there are real numbers
$\beta_{K,0},\ldots,\beta_{K,L-1}$ such that
\begin{equation}\label{eq:finite-comparison-residual}
 f_k^{(K)}=\sum_{q=0}^{L-1}\frac{\beta_{K,q}}{c_k^{2q+1}}
 +O_{p,K,L}(c_k^{-2L-1}),\qquad k>K.
\end{equation}
\end{lemma}

\begin{proof}
For the sequence in \eqref{eq:finite-comparison-sequence},
\[
 F_K(z)\coloneqq E_p(z)Q_K(z),\qquad
 Q_K(z)\coloneqq \prod_{j\leq K}\frac{\tau_j^2-z^2}{c_j^2-z^2}.
\]
Set $P(u)\coloneqq |\sin\pi u|^p$ and $W_K(x)\coloneqq x^2|F_K(x)|^p$.  By Fields's expansion of the gamma quotient in
\eqref{eq:A-m} (see \cite{Fields, Frenzen}) and the fact that $Q_K$ is a rational function of $x^2$, we get, for every $L$,
\begin{equation}\label{eq:finite-comparison-amplitude}
 W_K(x)=P(x-\gamma)\left\{
 \sum_{r=0}^{L}d_{K,r}x^{-2r}
 +O_{p,K,L}(x^{-2L-2})\right\}.
\end{equation}
We write $P(u)$ in terms of an absolutely convergent Fourier series,
\[
 P(u)=b_0+\sum_{m\geq1}b_m\cos(2\pi mu),
 \qquad b_m=O_p(m^{-p-1}).
\]
Since $P(0)=0$,
$P(u)=\sum_{m\geq1}b_m\{\cos(2\pi mu)-1\}$.  We set
$c=c_k=k+\gamma$ and use partial fractions to get
\begin{align*} \int_0^\infty\frac{\cos({2\pi  m}(x-c))-1}{c^2-x^2}\dd x
 = &\frac{\cos({2\pi  m} c)}{2c}\int_{\R}\frac{\cos({2\pi  m} x)}{c-x}\dd x
+\frac{\sin({2\pi  m} c)}{2c} \int_{\R}\frac{\sin({2\pi  m} x)}{c-x}\dd x \\ & +
\frac{\sin({2\pi  m} c)}{c} \int_0^\infty \frac{\sin({2\pi  m} x)}{x+c}\dd x,\end{align*}
where the first two integrals on the right-hand side are to be understood in the principal value sense. By contour integration, we find that they cancel each other. Hence, using $(x+c)^{-1}=\int_0^\infty {\ee}^{-s(x+c)}\dd s$, we find that
\begin{equation}\label{eq:periodic-integral-identity}
 \int_0^\infty\frac{\cos({2\pi  m}(x-c))-1}{c^2-x^2}\dd x
 =\frac{\sin({2\pi  m} \gamma)}{c}
 \int_0^\infty\frac{e^{-{2\pi  m}cs}}{1+s^2}\dd s.
\end{equation}
The finite identity
\[
 \frac1{1+s^2}=\sum_{q=0}^{L-1}(-1)^qs^{2q}
 +(-1)^L\frac{s^{2L}}{1+s^2}
\]
shows that the last integral equals
\[
 \sum_{q=0}^{L-1}\frac{(-1)^q(2q)!}{({2\pi  m}c)^{2q+1}}
 +O_L(({2\pi  m}c)^{-2L-1}).
\]
The estimate $\sum_m|b_m|\log(2+m)<\infty$ justifies termwise integration.
It follows that
\begin{equation}\label{eq:periodic-integral-expansion}
 \int_0^\infty\frac{P(x-\gamma)}{c_k^2-x^2}\dd x
 =\sum_{q=0}^{L-1}\frac{\ell_q}{c_k^{2q+2}}
 +O_{p,L}(c_k^{-2L-2})
\end{equation}
for suitable constants $\ell_q$. 

It remains to estimate the contribution from the higher order terms in
\eqref{eq:finite-comparison-amplitude}.  Set
\[
 U_q(x)\coloneqq x^{2q}W_K(x)-P(x-\gamma)
 \sum_{r=0}^{q}d_{K,r}x^{2(q-r)},
 \qquad
 H_q(c)\coloneqq \int_0^\infty\frac{U_q(x)}{c^2-x^2}\dd x.
\]
Then $U_q$ is integrable and
$U_{q+1}=x^2U_q-d_{K,q+1}P(x-\gamma)$. Writing  $U_q=\frac{c^2-x^2}{c^2}U_q+\frac{x^2}{c^2}U_q$, we therefore get
\begin{equation}\label{eq:moment-recursion}
 H_q(c)=\frac1{c^2}\left\{
 \int_0^\infty U_q(x)\dd x+d_{K,q+1}
 \int_0^\infty\frac{P(x-\gamma)}{c^2-x^2}\dd x+H_{q+1}(c)
 \right\}.
\end{equation}
$L$ iterations of \eqref{eq:moment-recursion}, together with
\eqref{eq:periodic-integral-expansion}, yield a polynomial of degree $L$ in $c^{-2}$ with remainder 
$O(c^{-2L-2})$, given that $H_L(c)=O(c^{-2})$. To see that the latter bound holds, we split the integral that defines $H_L$ as $\int_0^{c/2}+\int_{c/2}^{3c/2}+\int_{3c/2}^\infty$.
The first and third integrals are both $O(c^{-2})$ because  $U_L$ is integrable. To deal with the middle integral, we use
$U_L(c+u)=O(c^{-2}P(u))$ and get that it is $O(c^{-3} \log(2+c))$. 

Finally, since $W_K(x)$ is an even function,
\[
 f_k^{(K)}=2c_k\int_0^\infty\frac{W_K(x)}{c_k^2-x^2}\dd x,
\]
and \eqref{eq:finite-comparison-residual} follows.
\end{proof}

\subsection{Weighted sequence spaces}
For $r\geq0$, we write
\[
 \ell_r^2\coloneqq \Big\{\bm u=(u_j): \ \|\bm u\|_{r,2}^{{2}}\coloneqq \sum_{j=1}^\infty |u_j|^2 j^{2r} < \infty \Big\}.
\]
Given an admissible sequence $\bm t$, we define, for every $M\ge 0$, the real Hilbert space $\mathcal H_M(\mathbf t)$ in the following way. A real sequence $\bm u=(u_j)$ belongs to $\mathcal H_M(\mathbf t)$ if there exist coefficients $a_0,\ldots , a_{M}$ such that the sequence $\bm \nu =(\nu_j)$ defined by
\begin{equation}\label{eq:HM-definition}
     \nu_j \coloneqq u_j-\sum_{k=0}^{M} \frac{a_k}{t_j^{2k+1}} \end{equation}
belongs to $\ell^2_{2M+1}$. We set
\[ \|u\|_{\mathcal H_M(\mathbf t)}^2\coloneqq \sum_{k=0}^M |a_k|^2 + \| \bm \nu \|_{2M+1,2}^2.\]

The following lemma will be useful, as we shall repeatedly change between expansions in $t_k^{-1}$ and in
$c_k^{-1}$.

\begin{lemma}\label{lem:change-of-sequence-variable}
If $t_j-c_j=O(c_j^{-1})$, then
$\mathcal H_0(\mathbf t)=\mathcal H_0(\mathbf c)$ with equivalent norms.
If $M\geq1$ and
\begin{equation}\label{eq:change-variable-assumption}
 \mathbf t-\mathbf c\in\mathcal H_{M-1}(\mathbf c),
\end{equation}
then
\begin{equation}\label{eq:change-variable-conclusion}
 \mathcal H_M(\mathbf t)=\mathcal H_M(\mathbf c)
\end{equation}
with equivalent norms.  {The constants are uniform when $\inf_j t_j/c_j$ is bounded
away from zero and $\|\mathbf t-\mathbf c\|_{\ell^\infty_1}$ for $M=0$,
or $\|\mathbf t-\mathbf c\|_{\mathcal H_{M-1}(\mathbf c)}$ for $M\ge1$,
is bounded.}
\end{lemma}

\begin{proof}
The first assertion is trivial. If 
\eqref{eq:change-variable-assumption} holds, then the quotient
$u_j\coloneqq (t_j-c_j)/c_j$ is a finite linear combination of
$c_j^{-2},\ldots,c_j^{-2M}$ modulo $\ell_{2M}^2$.  Expanding
\[
 t_j^{-(2k+1)}=c_j^{-(2k+1)}(1+u_j)^{-(2k+1)}
\]
to the order ${c_j^{-2M-1}}$, we get an upper-triangular change of basis with diagonal
entries one. Clearly, the sum of all higher order terms belongs to
$\ell_{2M+1}^2(\mathbf c)$. {The stated bounds control the coefficients, the weighted
Taylor remainders, and the inverse triangular change of basis, giving
the asserted uniform norm estimates.}
\end{proof}

The only estimate for singular sums that we need is the following elementary
identity. Here we use the notation
\[ (K_{\bm t}\bm u)_k\coloneqq \sum_{j\neq k} \frac{u_j}{t_j^2-t_k^2}.\]

\begin{lemma}\label{lem:moment-subtraction}
Let $\bm t$ be an admissible sequence satisfying  $t_k=c_k+O(c_k^{-1})$. If $\bm q$ is in $\ell^2_{2L}$ with $L\ge 1$, then 
{the moments}
\[
 \mu_\nu\coloneqq \sum_{j\ge 1}t_j^{2\nu}q_j,
 \qquad 0\leq\nu<L,
\]
converge absolutely and
\begin{equation}\label{eq:moment-subtraction}
 t_k\sum_{j\ne k}
 \frac{q_j}{t_j^2-t_k^2}
 =-\sum_{\nu=0}^{L-1}\frac{\mu_\nu}{t_k^{2\nu+1}}
 +\frac{Lq_k}{t_k}
 +\frac{(K_{\mathbf t}u)_k}{t_k^{2L-1}},
 \qquad u_j=t_j^{2L}q_j.
\end{equation}
\end{lemma}

\begin{proof}
Absolute convergence of the moments follows from the Cauchy--Schwarz inequality.  Writing
\[
 \frac1{a-b}=-\sum_{\nu=0}^{L-1}\frac{a^\nu}{b^{\nu+1}}
 +\frac{a^L}{b^L(a-b)}
\]
with $a=t_j^2$ and $b=t_k^2$  summing over $j\neq k$, we get \eqref{eq:moment-subtraction}.
\end{proof}

We now show that the differential preserves the spaces
\eqref{eq:HM-definition}. In what follows, we use the notation $A\coloneqq D\mathscr Z(\bm t)$.

\begin{lemma}\label{lem:differential-expansions}
Let $\mathbf t$ be {in $\mathcal B(\boldsymbol\tau;\rho_0)$} satisfying
$t_k=c_k+O(c_k^{-1})$, and suppose that $\mathscr Z(\bm t)$ is in $\mathcal H_M(\mathbf t)$. 

If $M=0$ and $\bm v$ is in $\ell^\infty_1$, then
\begin{equation}\label{eq:offdiag-preserves-base}
 A^\circ v\in\mathcal H_0(\mathbf t).
\end{equation}
If $M\geq1$ and $\bm v$ is in $\mathcal H_{M-1}(\mathbf t)$, then
\begin{equation}\label{eq:offdiag-raises-order}
 A^\circ v\in\mathcal H_M(\mathbf t).
\end{equation}
In addition,
\begin{equation}\label{eq:diagonal-even-expansion}
 ((A_{kk}-\lambda_p)/t_k)
 \in \mathcal{H}_M(\bm t),
\end{equation}
with the coefficient $a_0$ of $(1/t_j)$ being $0$. {The operator bounds are uniform when
$\|\mathscr Z(\mathbf t)\|_{\mathcal H_M(\mathbf t)}$
and $\sup_k k|t_k-c_k|$ are bounded.}
\end{lemma}

\begin{proof}
By the assumption on $z=\mathscr Z \bm t$, we may write
\begin{equation}\label{eq:start}
 z_j=\sum_{a=0}^{M}b_at_j^{-(2a+1)}+\eta_j,
 \qquad \bm\eta\in\ell_{2M+1}^2(\mathbf t).
\end{equation}
Hence, to prove \eqref{eq:offdiag-preserves-base} and \eqref{eq:offdiag-raises-order}, it suffices to show that $A_0^\circ \bm v, \ldots, A_M^\circ\bm v$, and $A_{\bm \eta}^\circ \bm v$ are all in $\mathcal H_M(\bm t)$, where  
\[ (A_a^\circ \bm v)_k=2p \sum_{j\neq k}  \frac{t_jt_k^{-2a-1}-t_k t_j^{-2a-1}}{t_j^2-t_k^2} v_j, \qquad 0\le a \le M,\]
and
\begin{equation}\label{eq:off}  (A_{\bm \eta}^\circ \bm v)_k=2p \sum_{j\neq k}  \frac{t_j\eta_k-t_k \eta_j}{t_j^2-t_k^2} v_j. \end{equation}
Polynomial division in
\eqref{eq:Jacobian} gives
\begin{align}\label{eq:pure-power-action}
 (A_a^\circ \bm v)_k
 = 2p\sum_{\ell=0}^{a}t_k^{-(2(a-\ell)+1)}
 \sum_{j\ge 1} t_j^{-(2\ell+1)}v_j
 {-2p(a+1)t_k^{-(2a+2)}v_k}.
\end{align}
{For $M=0$ the last term belongs to $\ell^2_1$ since
$v_j=O(c_j^{-1})$. For $M\ge1$, expand $v\in\mathcal H_{M-1}(\bm t)$:
multiplication by $t_k^{-2a-2}$ gives the required odd powers and a
remainder in $\ell^2_{2M+1}$. Thus $A_a^\circ\bm v\in\mathcal H_M(\bm t)$.}

To compute the contribution from the remainder $\eta_j$, we rewrite
\eqref{eq:off} as 
\begin{equation}\label{eq:remainder-offdiag-action}
 (A^\circ_\eta \bm v)_k
 =2p\eta_k\sum_{j\ne k}
 \frac{t_jv_j}{t_j^2-t_k^2}
 -2pt_k\sum_{j\ne k}
 \frac{\eta_jv_j}{t_j^2-t_k^2}.
\end{equation}
In both cases under consideration, $(t_jv_j)$ is bounded.  The first sum is
therefore uniformly bounded, and hence the first term
belongs to $\ell_{2M+1}^2(\mathbf t)$.  To estimate the second term, we apply
Lemma~\ref{lem:moment-subtraction} with
$q_j=\eta_jv_j$ and $L=M+1$. Indeed, since the sequence $(t_j v_j) $ is bounded, we have $\bm q$ in $\ell^2_{2L}$.
Formula \eqref{eq:moment-subtraction} gives the required odd powers and a
remainder in $\ell_{2M+1}^2$ since $K_{\bm t}$ is bounded on $\ell^2$ by Hilbert's inequality.  This proves
\eqref{eq:offdiag-preserves-base} and \eqref{eq:offdiag-raises-order}.

To deal with the diagonal, we use \eqref{eq:diagonal-residual-formula}. We consider the contribution from each of the terms in \eqref{eq:start} separately. If we set $t_j^{-2a-1}$ in place of $z_j$ in the expression
\begin{equation} \label{eq:express}
 \sum_j\frac{z_j}{t_j}-t_k^2\sum_{j\ne k}
 \frac{z_k/t_k-z_j/t_j}{t_j^2-t_k^2}
\end{equation}
in \eqref{eq:diagonal-residual-formula}, then it becomes, by polynomial division,  
{$(a+1)t_k^{-2a-2}$ plus a linear combination of
$t_k^{-2},\ldots,t_k^{-2a}$; after division by $t_k$ the additional
power lies in the required expansion or its remainder.} To deal with the contribution from the remainder $\eta_j$, we set $\rho_j\coloneqq \eta_j/t_j$ which we insert in place of $z_j/t_j$ in \eqref{eq:express}:
\begin{equation}\label{eq:rho}
 \sum_j\rho_j-t_k^2\sum_{j\ne k}
 \frac{\rho_k}{t_j^2-t_k^2}+t_k^2\sum_{j\ne k}
 \frac{\rho_j}{t_j^2-t_k^2}.
\end{equation}
Since $\bm \rho$ is in $\ell_{2M+2}^2$, we may apply
Lemma~\ref{lem:moment-subtraction} with $L=M+1$ to the third sum in \eqref{eq:rho}. Its zeroth moment cancels the first sum
$\sum_j\rho_j$.  The remaining moments
yield a linear combination of $t_k^{-2},\ldots,t_k^{-2M}$. The middle sum in \eqref{eq:rho} is in $\ell_{2M}^2$, since the sum 
 $\sum_{j\neq k} 
 \frac1{|t_j^2-t_k^2|}$ is uniformly bounded.
We conclude that also \eqref{eq:diagonal-even-expansion} holds.
\end{proof}
In the next lemma, we will use the following notation:
\[ \mathcal U(R,\rho_0)\coloneqq \{\bm v: \ \bm v\in \mathcal{B}(\bm \tau, \rho_0), \|\mathscr Z \bm v \|_{{\ell^\infty_1}}\le R \}. \]

\begin{lemma}\label{prop:inverse-preserves-expansions}
We make the same assumptions as in Lemma~\ref{lem:differential-expansions} and let $R>0$. Then
\begin{equation}\label{eq:inverse-preserves-expansions}
 D\mathscr Z(\mathbf t)^{-1}:\mathcal H_M(\mathbf t)
 \longrightarrow\mathcal H_M(\mathbf t)
\end{equation}
is uniformly bounded {on subsets of $\mathcal U(R,\rho_0)$
where $\|\mathscr Z(\mathbf t)\|_{\mathcal H_M(\mathbf t)}$
and $\sup_k k|t_k-c_k|$ are bounded.}
\end{lemma}

\begin{proof}
We consider $\bm g=A\bm v$ in $\mathcal H_M(\mathbf t)$.  Lemma~\ref{prop:local-residual-inverse}
first shows that $\bm v$ satisfies $v_k=O(t_k^{-1})$.  When $M=0$,
Lemma~\ref{lem:differential-expansions} shows that 
$A^\circ \bm v$ is in $\mathcal H_0$ and
$(A_{kk}-\lambda_p)v_k$ is in $\ell_1^2$.  Therefore
\begin{equation}\label{eq:inverse-induction-identity}
 v=\lambda_p^{-1}
 \{g-A^\circ v-(A_{kk}-\lambda_p)v_k\}
 \in\mathcal H_0.
\end{equation}
Suppose now that the assertion is known at level $M-1$.  Since
$g,{z}\in\mathcal H_M(\bm t)\subset\mathcal H_{M-1}(\bm t)$, the induction hypothesis implies that
$v$ is in $\mathcal H_{M-1}(\bm t)$.  Lemma~\ref{lem:differential-expansions} puts
$A^\circ v$ in $\mathcal H_M(\bm t)$.  The even expansion
\eqref{eq:diagonal-even-expansion}, multiplied by
$v\in\mathcal H_{M-1}(\bm t)$, also belongs to $\mathcal H_M(\bm t)$.  For its weighted
remainder this is simply
\[
 t_k^{2M+1}|\delta_kv_k|
 \leq |t_k^{2M}\delta_k|\,|t_kv_k|.
\]
Formula \eqref{eq:inverse-induction-identity} completes the induction.  The
same estimates give the asserted uniform bounds.
\end{proof}

\subsection{Proof of Theorem~\ref{thm:main}}


We choose $K$ so large that the sequence $\mathbf a^{(K)}$ in
\eqref{eq:finite-comparison-sequence} belongs to $\mathcal B(\bm \tau ,\rho_0)$ and the line
segment from $0$ to $f^{(K)}$ is contained in $\mathscr Z \mathcal B(\bm \tau ,\rho_0)$, with $\rho_0$ as in Lemma~\ref{prop:local-residual-inverse}.  This is possible
because $\mathbf a^{(K)}\to\boldsymbol\tau$ in $\ell^2$ and
$f^{(K)}\to0$.  We fix this $K$ 
and define
\begin{equation}\label{eq:residual-path}
 \mathbf t(s)\coloneqq \mathscr Z^{-1}(sf^{(K)}),\qquad0\leq s\leq1.
\end{equation}
Then
\begin{equation}\label{eq:residual-path-endpoints}
 \mathbf t(0)=\boldsymbol\tau,\qquad
 \mathbf t(1)=\mathbf a^{(K)},\qquad
 \dot{\mathbf t}(s)\coloneqq \frac{\dd}{\dd s}\bm{t}(s)=D\mathscr Z(\mathbf t(s))^{-1}f^{(K)}.
\end{equation}
By Lemma~\ref{lem:finite-comparison-residual},
$f_k^{(K)}=O(c_k^{-1})$.  The last part of
Lemma~\ref{prop:local-residual-inverse} therefore gives
\[
 \dot{\bm t}_k(s)=O(c_k^{-1})
\]
uniformly for $0\leq s\leq1$.  Since $t_k(1)=c_k$ for $k>K$, it follows
that
\begin{equation}\label{eq:rough-path-estimate}
 t_k(s)-c_k=O(c_k^{-1})
\end{equation}
uniformly on the path. Therefore, since $c_{k+1}-c_k=1$, there exists a positive integer $N$ such that $t_{k+1}(s)-t_k(s)\ge \tfrac12$ when $k>N$. 
Each of the functions $s\mapsto t_{k+1}(s)-t_k(s)$ is continuous and positive  on $[0,1]$. Hence, by compactness of $[0,1]$, we have the required uniform separation $\inf_{0\le s \le 1}\inf_{k\ge 1}(t_{k+1}(s)-t_k(s))>0$.

Lemma~\ref{lem:finite-comparison-residual} also gives
\begin{equation}\label{eq:comparison-in-all-HM}
 f^{(K)}\in\mathcal H_M(\mathbf c)
 \qquad(M=0,1,2,\ldots).
\end{equation}
By \eqref{eq:rough-path-estimate} and
Lemma~\ref{lem:change-of-sequence-variable}, the spaces
$\mathcal H_0(\mathbf t(s))$ and $\mathcal H_0(\mathbf c)$ agree uniformly.
Since $\mathscr Z(\mathbf t(s))=sf^{(K)}$,
Lemma~\ref{prop:inverse-preserves-expansions} gives
that $\dot{\mathbf t}(s)$ is in $\mathcal H_0(\mathbf c)$
uniformly. Writing
\begin{equation} \label{eq:integral}
 \mathbf t(s)-\mathbf c=\mathbf a^{(K)}-\mathbf c
 -\int_s^1\dot{\mathbf t}(u)\dd u ,
\end{equation}
we conclude that the whole path, and in particular
$\bm\tau-\bm c$, belongs to $\mathcal H_0(\bm c)$.

Suppose inductively that
$\mathbf t(s)-\mathbf c$ is in $\mathcal H_{M-1}(\mathbf c)$ uniformly on the
path. Then Lemma~\ref{lem:change-of-sequence-variable} implies that
$\mathcal H_M(\mathbf t(s))=\mathcal H_M(\mathbf c)${ with uniform norm equivalence.
In particular, $\mathscr Z(\mathbf t(s))=sf^{(K)}$ is uniformly bounded
in $\mathcal H_M(\mathbf t(s))$, so the strengthened inverse bounds apply.} Equations
\eqref{eq:comparison-in-all-HM} and
\eqref{eq:residual-path-endpoints}, together with
Lemma~\ref{prop:inverse-preserves-expansions}, give that
$\dot{\mathbf t}(s)$ is in $\mathcal H_M(\mathbf c)$
with uniform control of norms. Resorting again to \eqref{eq:integral}, we find that {$\mathbf t(s)-\bm c$ is uniformly bounded} in $\mathcal H_M(\mathbf c)$.
{Here the coefficient projections are limits of coordinate
expressions. The paths are therefore strongly measurable in the separable
spaces $\mathcal H_M(\bm c)$, and the uniform bounds justify the integrals.}

The preceding argument establishes Theorem~\ref{thm:main}. Indeed, the coefficients $a_k$ in \eqref{eq:HM-definition} of a vector $\bm u$ in $\mathcal H_M(\bm c)$ are clearly unique. Since $\mathcal H_{M+1}(\bm c)\subset\mathcal H_M(\bm c)$, the coefficients found
at successive levels agree. This means that if we let
$[u]_{q,\mathbf c}$ denote the coefficient of $c^{-(2q+1)}$ in the unique
decomposition of $u$ in $\mathcal H_q(\mathbf c)$, then we have identified the coefficients in the expansion \eqref{eq:complete-zero-Poincare} as
\begin{equation}\label{eq:coefficient-path-formula}
 a_{p,q}=-\int_0^1
 \left[D\mathscr Z(\mathbf t(s))^{-1}f^{(K)}\right]_{q,\mathbf c}\dd s.
\end{equation}
By uniqueness of the expansion, the latter integral does not depend on our choice of $K$.

\section{Fourier regularity}\label{sec:Fourier-regularity}
\begin{proof}[Proof of Corollary~\ref{cor:Fourier}]
Put $\alpha=2/p$ and $r_p=\varphi_p-K(p)E_p$, with $K(p)$ as in the
introduction.  Theorem~\ref{thm:main} gives
$\Delta_n\coloneqq\tau_n^2-c_n^2=O_p(1)$; for $p=1$ the same conclusion
follows from the zero expansion of \cite{BORS} recalled in the introduction.
For $x\in I_m=[c_m-1/2,c_m+1/2]$ and sufficiently large $m$, set
\[
 P_m(x)=\prod_{n\ne m}
 \left(1+\frac{\Delta_n}{c_n^2-x^2}\right).
\]
Splitting the sum at $m/2$ and $2m$ gives, uniformly on $I_m$,
\[
 \sum_{n\ne m}\left|\frac{\Delta_n}{c_n^2-x^2}\right|
 \le C_p\frac{\log(2+m)}m,
 \qquad
 P_m(x)=1+O_p\!\left(\frac{\log(2+m)}m\right).
\]
The exact product identity is
\[
 r_p(x)=K(p)\left[
 E_p(x)(P_m(x)-1)
 +\Delta_m\frac{E_p(x)}{c_m^2-x^2}P_m(x)\right].
\]
Lemma~\ref{lem:Gamma-properties} yields $|E_p(x)|\le C_pm^{-\alpha}$
and $|E_p(x)/(c_m^2-x^2)|\le C_pm^{-\alpha-1}$, with the latter quotient
continued at $x=c_m$.  By evenness and boundedness on compact intervals,
\begin{equation}\label{eq:physical-rough}
 |r_p(x)|\le C_p
 \frac{\log(2+|x|)}{(1+|x|)^{\alpha+1}}.
\end{equation}
Consequently $(1+x^2)^{s/2}r_p(x)\in L^2(\R)$ for every
$s<\alpha+1/2$.  Weighted Plancherel therefore gives
$h_p\coloneqq\widehat r_p\in\bigcap_{s<\alpha+1/2}H^s(\R)$.
Its support lies in $[-\pi,\pi]$ by the Paley--Wiener theorem, and the
Fourier transform of $E_p$ displayed in the introduction gives the stated
coefficient $A(p)$.
\end{proof}

\subsection*{LLM usage disclosure:} The work on this project started as a natural continuation of \cite{BCOS}. The authors provided a roadmap to tackle the problem using the scheme previously used in \cite{BORSprelim}. After this, they used a combination of the LLMs GPT-5.6 Sol, Claude Fable 5, and  later GPT-6 Astra to fill in the technical details, and provide mathematical feedback. The results were then rewritten by the authors who take full responsibility for their correctness.

\end{document}